\documentclass[10pt,a4paper,oneside]{amsproc}
\usepackage[utf8]{inputenc}
\usepackage{latexsym, amsthm}
\usepackage{amsfonts, amsmath, amssymb,comment, mathtools}
\usepackage{euscript,mathrsfs}
\usepackage{url}
\usepackage{array}
\usepackage[a4paper]{geometry}
\usepackage{graphics}
\usepackage[usenames]{color}
\usepackage{tikz-cd}

\DeclareFontFamily{U}{mathx}{}
\DeclareFontShape{U}{mathx}{m}{n}{<-> mathx10}{}
\DeclareSymbolFont{mathx}{U}{mathx}{m}{n}
\DeclareMathAccent{\widehat}{0}{mathx}{"70}
\DeclareMathAccent{\widecheck}{0}{mathx}{"71}

\newtheorem{theorem}{Theorem}[section]

\newtheorem{example}[theorem]{Example}
\newtheorem{conjecture}{Conjecture}

\newtheorem{question}[theorem]{Question}
\newtheorem{definition}[theorem]{Definition}
\newtheorem{remark}[theorem]{Remark}

\newtheorem{lemma}[theorem]{Lemma}

\newtheorem{proposition}[theorem]{Proposition}

\newcommand{\B}{\mathcal{B}}

\newcommand{\ZZ}{\mathbb{Z}}

\newcommand{\Fq}{\mathbb{F}_q}

\newcommand{\M}{\mathcal{M}}
\newcommand{\F}{\mathfrak{F}}

\newcommand{\Omb}{\Omega_\alpha (\ell, V_{m}, \B)}
\newcommand{\Ombb}{\Omega_{\alpha'} (\ell, V_{m-1}, \B')}
\newcommand{\Ombh}{\Omega_{\widecheck{\alpha}} (\ell -1, V_{m-1}, \B')}
\def\a{{\alpha}}
\def\Fq{{\mathbb F}_q}

\def\FF{{\mathbb F}}
\def\PP{{\mathbb P}} 
\newcommand{\V}{{\mathcal{V}}}

\begin{document}
\title{Towards the characterization of minimum weight codewords of Schubert codes}  

\author{Mrinmoy Datta}

\address{Department of Mathematics, \newline \indent
Indian Institute of Technology Hyderabad, Kandi, Sangareddy, Telangana, India}
\email{mrinmoy.datta@math.iith.ac.in}
\thanks{The first named author is partially supported by a research grant 02011/34/2025/R\&D-II/DAE/9465 from the National Board for Higher Mathematics, Department of Atomic Energy, India.} 

\author{Tiasa Dutta}

\address{Department of Mathematics, \newline \indent
Indian Institute of Technology Hyderabad, Kandi, Sangareddy, Telangana, India}
\email{duttatiasa98@gmail.com}
\thanks{The second named author is partially supported by a DST-INSPIRE Ph.D. fellowship from the
Department of Science and Technology, Govt. of India}

\author{Trygve Johnsen}
\address{Department of Mathematics and Statistics, University of Tromsø (UiT), Norway}
\email{trygve.johnsen@uit.no }
\thanks{The third named author was supported by the Tromsø Research Foundation (project “Pure Mathematics in Norway”) and the  UiT Aurora project MASCOT}



\date{}
\begin{abstract}   
A formula for the minimum weight of Schubert codes was conjectured by Ghorpade and Lachaud in 2000. This was established by Xiang in 2008. In 2018, Ghorpade and Singh provided a new proof of this conjecture. Moreover, they also conjectured that the minimum weight codewords of the Schubert codes $C_\a (\ell, m)$ are given by the so-called \emph{Schubert decomposable codewords}. We prove the validity of the conjecture for all Schubert varieties $\Omega_{\a}(\ell, V_m)$ for all but finitely many values of $q$.
\end{abstract}
\maketitle

\section{Introduction}
Let $\Fq$ be a finite field with $q$ elements, where $ q$ is a prime power. For positive integers $n$ and $k$ with $k \le n$, an $[n, k]_q$ linear code (or, simply an $[n, k]$-code) $C$ is a $k$-dimensional subspace of the vector space $\Fq^n$. The parameters $n$ and $k$ are called the length and the dimension of the code $C$, respectively.  Elements of a code are called codewords. The vector space, with usual coordinate system, can be endowed with a metric, known as the \textit{Hamming metric}, denoted by $d_H$, where for any two elements ${\bf a} = (a_1, \dots, a_n) \in \Fq^n$ and ${\bf b} = (b_1, \dots, b_n) \in \Fq^n$,  their  \emph{Hamming distance} is given by
$$d_H ({\bf a}, {\bf b}) = |\{i : a_i \neq b_i\}|.$$
Furthermore, the Hamming distance naturally gives rise to the notion of \emph{Hamming weight} of a codeword. Given an element ${\bf c} = (c_1, \dots, c_n) \in \Fq^n$, we define its \emph{Hamming weight}, denoted by $$w_H ({\bf c}) = |\{ i : c_i \neq 0\}|.$$
We define the \emph{minimum distance of a code $C$}, denoted by $d (C)$, as 
$$d(C) = \min \{d({\bf x}, {\bf y}) : {\bf x}, {\bf y} \in C,  {\bf x} \neq  {\bf y} \}.$$
Due to the linearity of the code $C$, it follows that the minimum distance of a code $C$ is equal to the \emph{minimum weight} of $C$, where the latter is defined as 
$$\min \{w_H ({\bf c}) : {\bf c} \in C, {\bf c} \neq 0\}.$$
In order to produce good codes, one often takes recourse to algebraic varieties over finite fields and codes that occur from them. One way to construct codes from algebraic varieties is to use the notion of a projective system. We omit a detailed discussion of projective systems and codes here, but refer the reader to \cite{TVN1}. 

The codes from Schubert varieties, also known as Schubert codes, were first introduced in \cite{GL} by Ghorpade and Lachaud. Let $\a = (\a_1, \dots, \a_{\ell}) \in I(\ell, m)$, as defined in \eqref{ilm}, $V_m$ a vector space of dimension $m$ over $\Fq$ with an ordered basis $\B$ and $\Omb$ be the Schubert subvariety of the Grassmannian $G(\ell, V_m)$ embedded in a nondegenerate manner in the projective space $\PP^{k_\a - 1}$ as described in Subsection \ref{sec:sch}. It turns out that the length $n_\a$, dimension $k_\a$, and the minimum distance/weight $d_\a$ of the code $C_{\a} (\ell, V_m, \B)$ arising from the Schubert subvariety $\Omb$ are given by the following formulas:
\begin{enumerate}
    \item $n_\a = |C_{\a} (\ell, V_m, \B)|$. We refer to \cite[Eq. (9), (10) and Theorem 1]{GT} for combinatorial formulas describing $n_\a$. We also discuss this matter at length in Subsection \ref{subsec:en} later. 
    \item $k_\a$ is given in \cite[Theorem 7]{GT}.
    \item $d_\a = n_\a - \max \{|\Pi_\a \cap \Omb| : \Pi_\a \ \text{is a hyperplane of $\PP^{k_\a - 1}$} \} = q^{\delta (\a)}$, where $\delta (\a) = (\a_1 + \dots + \a_\ell) - \frac{\ell (\ell + 1)}{2}$. This formula was first conjectured in \cite{GL}, and subsequently proved in \cite{C} and \cite{GV} in the case $\ell = 2$, and in full generality in \cite{X}. An independent proof for this formula was also given in \cite{GS}.
\end{enumerate}
We note that when $\a = (m - \ell + 1, \dots, m)$, the Schubert subvariety $\Omb$ equals the Grassmannian $G(\ell, V_m)$. In this particular case, all the above formulas are known:
$$n_\a  = {m \brack \ell}_q, \ k_\a = {m \choose \ell} \ \text{and} \ d_\a = q^{\ell (m - \ell)}.$$
While the first two formulas are classical, the minimum distance of this particular Schubert code, namely the Grassmann code, was determined by Nogin \cite{N}. Moreover, Nogin also provided a complete classification of the minimum-weight codewords of Grassmann codes. Although, in general, the minimum weight of the Schubert codes is known, a classification of all the minimum weight codewords is yet to be done. In \cite{GS}, the authors conjectured a classification of all the minimum weight codewords of Schubert codes (cf. Conjecture \ref{conj}). Their conjecture states that the minimum-weight codewords of Schubert codes are given by Schubert decomposable hyperplanes (cf. Definition \ref{schdec}). Ghorpade and Singh proved the conjecture in the affirmative in the case when $\ell = 2$. In this paper, we show that the conjecture is true for all positive integers $\ell$, when $q$ is sufficiently large. In achieving this, we also give an alternative proof of the formula for $d_\a$ as mentioned above.

This paper is organized as follows: In Section 2, we gather all the preliminary definitions and results that are needed to prove our main result. This includes a detailed discussion on Grassmann varieties and their Schubert subvarieties. In Section 3, we provide several combinatorial decompositions of Schubert subvarieties and a discussion of the enumeration of points on Schubert subvarieties over finite fields. Finally, in Section 4, we prove our main result which is given in Theorem \ref{main}.

\section{Preliminaries}\label{sec:prel}
In this section, we recall some of the properties of Grassmannians and their Schubert subvarieties. Our main goal is to record the well-known properties of these objects that will be used later in this paper. None of the results in this section are new. They can be found in \cite{A, HP, M}, and in modern texts such as \cite{LB}. Moreover, we borrow some definitions and results in coding theory from several sources \cite{C, GL, GPP, GS, HJR, N, X}. We also refer  to \cite[Section 2]{DD} and the references therein for a more detailed understanding of the material covered in this section.

\subsection{Grassmannian, Pl\"ucker projective space, and Pl\"ucker embedding}\label{sec:gra}
Fix integers $\ell, m$ satisfying $1 \le \ell \le m$ and an $m$-dimensional vector space $V_m$ over a field $\mathbb{F}$ with an ordered basis $\mathcal{B} = \{v_1, \dots, v_m\}$ of $V_m$. 
The \emph{Grassmannian of all $\ell$-dimensional subspaces of $V_m$}, denoted by  $G (\ell, V_m)$ is defined as
$$G (\ell, V_m) :=\{ L \subset V_m : L \ \text{is a subspace of} \ V_m, \ \dim L = \ell\}.$$
Let us now
impose a geometric structure on $G(\ell, V_m)$ via the so-called \textit{Pl\"{u}cker embedding} that is given by:
\begin{equation}\label{pl1}
    \pi_{\ell, m} : G(\ell, V_m) \to \PP \left(\bigwedge^\ell V_m\right) \ \ \text{given by} \ \ L \mapsto [\omega_1 \wedge \cdots \wedge \omega_\ell],
\end{equation}
where $\{\omega_1, \dots, \omega_\ell\}$ is a basis of $L$. It is well-known that the map is well-defined, that is, it is independent of the choice of basis for the elements of $G(\ell, V_m)$. The projective space $\PP \left(\bigwedge^\ell V_m\right)$ is known as the \textit{Pl\"ucker projective space}. Let
\begin{equation}\label{ilm}I(\ell, m) := \{\a = (\a_1, \dots, \a_\ell) \in \ZZ^\ell : 1 \le \a_1 < \cdots < \a_\ell \le m\}.\end{equation}
Note that $\bigwedge^{\ell} V_m$ is a $k = {m \choose \ell}$-dimensional vector space over $\mathbb{F}$. Moreover, a fixed ordering of elements of $I(\ell,m)$ gives an ordered basis $\{v_\a = v_{\a_1} \wedge \cdots \wedge  v_{\a_\ell} : \a \in I(\ell, m)\}$ of $\bigwedge^{\ell} V_m$. This ordered basis, in turn, helps us in introducing the homogeneous coordinates of the Pl\"ucker projective space. In particular, if $L \in G(\ell, V_m)$ and $A_L = (a_{ij})$ is an $\ell \times m$ matrix with entries in $\FF$ such that the row-space of $A_L$ is $L$, then 
\begin{equation}\label{homcoor}
    \pi_{\ell, m} (L) = [p_\a (A_L)]_{\a \in I(\ell, m)}
\end{equation}
where $p_{\a} (A_L)$ is the $\a$-th minor of $A_L$.
As mentioned above, it is known (cf. \cite[Section 5.2.1]{LB}) that $\pi_{\ell, m} (L)$ is independent of the choice of $A_L$ and consequently the map $\pi_{\ell, m}$ is well-defined. Moreover, the map  $\pi_{\ell, m}$ is injective \cite[Theorem 5.2.1]{LB}.  We remark that the rows of $A_L$ above are written with respect to the fixed basis $\B = \{v_1, \dots, v_m\}$ of $V_m$. However, to our advantage, it turns out that changing the basis of $V_m$, does not affect the situation drastically. More precisely, a change of basis of $V_m$ induces a collineation of $\PP(\bigwedge^\ell V_m)$ taking $G(\ell, V_m)$ to itself. 
When $\FF$ is an algebraically closed field, it can be shown that $G(\ell, V_m)$ is given by the solutions to a system of some homogeneous quadratic equations in $X_\a$-s, known as the Plücker relations. We refer to \cite[Theorem 5.2.3]{LB} for a complete proof of this fact. Consequently, the subset $G(\ell, V_m)$ can be regarded as a projective algebraic variety when working over an algebraically closed field.

We keep working with the fixed ordered basis $\B = \{v_1, \dots, v_m\}$ of $V_m$, as mentioned above. For every $L \in G(\ell, V_m)$, we may choose a basis of $L$ in a way so that $A_L$ is in \emph{right-row-reduced-echelon} form, which is uniquely determined by $L$. That is, $L$ can be uniquely represented by an $\ell \times m$ matrix  $M_{\B}(L)$ satisfying
\begin{enumerate}
    \item[(a)] the rows of $M_{\B}(L)$ are elements of $L$ written with respect to the basis $\mathcal{B}$ of $V_m$.
    \item[(b)]  the row-space of $M_{\B}(L)$ is $L$, 
    \item[(c)] for each $i = 1, \dots, \ell$ the last non-zero entry of the $i$-th row, called the \emph{pivot of $i$-th row}, is equal to $1$,
    \item[(d)] the last non-zero entry of $(i+1)$-st row appears to the right of the last non-zero entry of the $i$-th row,
    \item[(e)] all the entries above and below of a pivot are $0$. 
\end{enumerate}
 In particular,  we have a natural one-to-one correspondence,
\begin{equation}\label{em}
    G(\ell, V_m) \longleftrightarrow \M (\ell, m) \ \ \text{given by} \ \ L \longleftrightarrow M_{\B}(L),
\end{equation}
where $\M (\ell, m)$ is the set of all $\ell \times m$ matrices in right-row-reduced echelon form with entries in $\FF$. 
Moreover, with respect to a fixed basis $\B$ of $V_m$, the map $\pi_{\ell, m}$ in \eqref{pl1} can be rewritten as 
\begin{equation}\label{pl2}
    \pi_{\ell, m}: G (\ell, V_m) \to \PP\left(\bigwedge^\ell V_m\right) \ \ \text{as} \ \ \pi_{\ell, m} (L) = [p_\a (M_{\B}(L))]_{\a \in I(\ell, m)}.
\end{equation}
This description will be used later in this article. 

\subsection{Schubert subvarieties of $G(\ell, V_m)$}\label{sec:sch} The Grassmannian $G(\ell, V_m)$ contains a special class of subvarieties, known as the \emph{Schubert subvarieties}, which are the objects of central interest in this article. For $\a \in I(\ell, m)$, the \emph{$\a$-th Schubert cell}, with respect to the basis $\B =\{v_1, \dots, v_m\}$, denoted by $C_{\a}(\ell, V_m, \B) $, 
is defined as,
$$C_{\a}(\ell, V_m, \B) := \{L \in G(\ell, V_m) : \ \text{the pivots of} \ M_{\B}(L) \ \text{are on the columns} \ \a_1, \dots, \a_\ell \}.$$
Note that $C_\a(\ell, V_m, \B)$ is identified with an affine space  of dimension $\delta (\a)$ over $\FF$, where $$\delta (\a) = (\a_1 + \dots + \a_\ell) - \frac{\ell (\ell+1)}{2}.$$
We define the partial order, also known as the Bruhat order, on the elements of $I(\ell, m)$ as follows: For $\a = (\a_1, \dots, \a_{\ell}), \beta = (\beta_1, \dots, \beta_{\ell}) \in I(\ell, m)$,  we say that 
$$\a \le \beta \iff \a_i \le \beta_i \ \ \ \text{for all} \ i \in \{1, \dots, \ell\}.$$ For $\a \in I(\ell, m)$, define $$\Delta (\a) = \{\beta \in I(\ell, m) : \beta \not\le \a\}, \ \nabla (\a) = I(\ell, m) \setminus \Delta (\a), \ \text{and} \ k_\a = |\nabla(\a)|.$$ 
The \emph{$\a$-th Schubert subvariety in} $G(\ell, V_m)$ with respect to the basis $\B$, denoted by $\Omb$, is defined as 
\begin{equation}\label{Sch}
\Omega_\a (\ell, V_m, \B):= \bigcup_{\beta \in \nabla(\a)} C_{\beta} (\ell, V_m, \B).
\end{equation}
Indeed, the union mentioned above is clearly a disjoint union of Schubert cells. Evidently, if $\a = (m - \ell + 1, \dots, m)$, then $\Omega_\a (\ell, V_m, \B) = G(\ell, V_m)$. 
For $i= 1, \dots, m$, let $V_i$ be the subspace of $V_m$ spanned by $\{v_1, \dots, v_i\}$. We see that for every $\a = (\a_1, \dots, \a_\ell) \in I(\ell, m)$, there is a partial flag of subspaces of $V_m$ given by
$$0 \subset V_{\a_1} \subset \cdots \subset V_{\a_\ell} \subset V_m \ \ \ \text{with} \ \ \ \dim V_{\a_i} = \a_i \ \ \text{for all} \ i =1, \dots, \ell.$$
Moreover, it follows rather trivially that
$$L \in C_{\a} (\ell, V_m, \B) \iff \dim (L \cap V_{\a_j}) = j \ \ \text{and} \ \ 
\dim (L \cap V_{i}) < j \ \ \text{for} \ i < \a_j, \ \text{for every} \ j = 1, \dots, \ell,$$ 
while  
$$L \in \Omega_{\a} (\ell, V_m, \B) \iff \dim (L \cap V_{\a_j}) \ge j \ \  
\text{for every} \ j = 1, \dots, \ell.$$

In the literature, several authors define a Schubert subvariety $\Omb$ in terms of the above equivalence condition. 
The Pl\"ucker embedding $\pi_{\ell, m}$ restricted to $\Omega_\a (\ell, V_m, \B)$ allows us to view the Schubert variety as a subset of the projective space $\PP\left(\bigwedge^{\ell} V_m\right)$. As mentioned above, with respect to the basis $\B$ and a fixed ordering of the elements of $I(\ell, m)$, the projective space $\PP^{k-1} = \PP\left(\bigwedge^{\ell} V_m\right)$ is endowed with the homogeneous coordinates $(X_{\a})_{\a \in I(\ell, m)}$. We define the projective linear subspace $\PP^{k_\a - 1}$ of $\PP\left(\bigwedge^{\ell} V_m\right)$ by 
\begin{equation}\label{pkalpha}
\PP^{k_\a - 1} := \left\{ P \in \PP\left(\bigwedge^{\ell} V_m\right) : X_\beta (P) = 0 \ \text{for all} \ \beta \in \Delta (\a)\right\}.
\end{equation}
It is well-known that 
\begin{equation}\label{commd1}
    \pi_{\ell, m} (\Omega_\a (\ell, V_m, \B))  = G(\ell, V_m) \cap \PP^{k_\a - 1},
\end{equation}
where $\PP^{k_\a - 1}$ is as defined above. As with the Grassmann varieties, we will denote by $\Omega_\a (\ell, V_m, \B)$, the Plücker embedding of $\Omega_\a (\ell, V_m, \B)$ in $\PP^{k_\a - 1}$.
We see that the $\a$-th Schubert variety $\Omb$ can be geometrically understood as a linear section of the Grassmann variety $G(\ell, V_m)$ given by some coordinate hyperplanes of the Pl\"ucker space.

\subsection{Hyperplanes, their decomposability, and Schubert decomposability}\label{sec:hyp}  
One of our main goals in this article is to study the maximum number of points lying on a \emph{hyperplane} section of a Schubert variety $\Omega_\a (\ell, V_m, \B)$ defined over a finite field. That is, we want to determine the quantity
$$\max \{|\Pi_\a \cap \Omega_\a (\ell, V_m, \B)|: \Pi_\a \ \text{is a hyperplane of} \ \PP^{k_\a - 1} \},$$
when all the subsets mentioned above are defined over a finite field $\Fq$. 
To this end, it is imperative to first understand the algebraic and geometric properties of hyperplanes in the Plücker space. We will briefly recall their definitions and properties in this subsection, but refer the reader to \cite[Section 2.3]{DD} and the references therein for more leisurely reading. 

The hyperplanes of the Pl\"ucker space $\PP(\bigwedge^\ell V_m)$ can be represented in two ways. On one hand, they are given by elements of the dual vector space $(\bigwedge^\ell V_m)^*$ which is known to be isomorphic to $(\bigwedge^{m -\ell} V_m)$. Thus a hyperplane of $\PP(\bigwedge^\ell V_m)$ can be defined by an element $F$, given by 
$$F = \sum_{\a \in I(m - \ell, m)} c_\a v_{\a_1} \wedge \cdots\wedge v_{\a_{m - \ell}} \in \bigwedge^{m-\ell} V_m.$$
As usual, a point $L \in G(\ell, V_m)$ lies on the hyperplane given by $F$ if and only if $F(L) =0$.
Here
$$F(L) = \sum_{\a \in I(m - \ell, m)} c_\a v_{\a_1} \wedge \cdots\wedge v_{\a_{m - \ell}} \wedge (w_1 \wedge \dots \wedge w_\ell),$$
where $\{w_1, \dots, w_\ell\}$ is an ordered basis of $L$. On the other hand, we may also describe a hyperplane of $\PP(\bigwedge^\ell V_m)$ geometrically in terms of the homogeneous coordinates $X_\a$ as mentioned earlier. Here, we may represent a hyperplane as the vanishing set of a homogeneous linear polynomial 
$$\sum_{\a \in I(\ell, m)} c_\a X_\a.$$

For each $\a = (\a_1, \dots, \a_\ell) \in I(\ell, m)$, let $\a^{\mathsf{C}} = (\a'_1, \dots, \a'_{m-\ell}) \in I(m-\ell, m)$ denote the unique element satisfying $\{\a'_1, \dots, \a'_{m-\ell}\} \cup \{\a_1, \dots, \a_\ell\} = \{1, \dots, m\}$, and let $\epsilon(\a)$ denote the sign of the permutation taking the tuple $(\a^{\mathsf{C}}, \a)$, obtained by juxtaposing $\a^{\mathsf{C}}$ and $\a$, to $(1, \dots, m)$. It is well-known (cf. \cite[Section 2.3]{DD}) that the bilinear map 
$$\wedge : \bigwedge^{m-\ell} V_m \times \bigwedge^{\ell} V_m \to \bigwedge^m V_m \cong \FF$$ 
is non-degenerate, and that the dual basis of $\{v_\a : \a \in I(\ell,m)\}$ with respect to this pairing is 
$\{\epsilon(\a)\, v_{\a^{\mathsf{C}}} : \a \in I(\ell, m)\}$. In particular, the coordinate hyperplane given by $X_\a$ corresponds to the hyperplane given by the element $\epsilon(\a)\, v_{\a^{\mathsf{C}}} \in \bigwedge^{m-\ell} V_m$, and this 
correspondence extends to the identification 
\begin{equation}\label{hypid}
    \sum_{\a \in I (\ell, m)} c_\a X_{\a} \longleftrightarrow 
    \sum_{\a \in I(\ell, m)} \epsilon(\a)\, c_\a\, v_{\a^{\mathsf{C}}}.
\end{equation}

The non-degeneracy of the Grassmannian $G(\ell, V_m)$ in $\PP(\bigwedge^\ell V_m)$ implies that if $\sum_{\a \in I(\ell, m)} c_\a X_\a (L) =0$ for all $L \in G(\ell, V_m)$, then $c_\a = 0$ for all $\a \in I(\ell, m)$. 

Let us now consider the $\a$-th Schubert subvariety $\Omb$. We recall that $\Omega_\a (\ell, V_m, \B) = G(\ell, V_m) \cap \PP^{k_\a - 1}$, where $\PP^{k_{\a} - 1}$ is the linear subspace of $\PP^{k-1}$ as given in \eqref{pkalpha}. Take a hyperplane $\Pi$ of $\PP (\bigwedge^{\ell} V_m)$ given by $F (X_\a : \a \in I(\ell, m)) = \sum_{\beta \in I(\ell, m)} c_\beta X_\beta$. We may write 
$$F = F_{\Delta (\a)} + F_{\nabla (\a)}, \ \ \text{where} \ \ F_{\Delta (\a)} = \sum_{\beta \in \Delta (\a)} c_\beta X_\beta \ \ \text{and} \ \ F_{\nabla (\a)} = \sum_{\beta \in \nabla (\a)} c_\beta X_\beta.$$
We see that a hyperplane $\Pi$ restricts to a hyperplane $\Pi_\a$ in $\PP^{k_\a - 1}$ if $F_{\nabla (\a)} \neq 0$.  Indeed the restricted hyperplane $\Pi_\a$ is defined as the set of zeroes of $F_{\nabla (\a)}$ in $\PP^{k_\a - 1}$. 
Moreover, it is well-known (cf. \cite[Remark 5.3.4]{LB}) that the $\a$-th Schubert variety $\Omb$ of $G(\ell, V_m)$ is also a non-degenerate subset of $\PP^{k_\a - 1}$. Consequently, the hyperplane $\Pi$ described above contains the Schubert variety $\Omb$ if and only if $F_{\nabla (\a)} = 0$. While a hyperplane of the Pl\"ucker space restricts \emph{uniquely} to a hyperplane of $\PP^{k_\a - 1}$, a hyperplane of $\PP^{k_\a - 1}$ can be \emph{extended} to a hyperplane of the Pl\"ucker space in several ways. That is, if $\Pi_\a$ is a hyperplane of $\PP^{k_\a - 1}$ given by the vanishing set of $F_\a = \sum_{\beta \in \nabla (\a)} c_\beta X_\beta$, then one can construct a hyperplane $\Pi$ of $\PP (\bigwedge^\ell V_m)$ given by $F = F_\a + \sum_{\gamma \in \Delta (\a)} c_\gamma X_\gamma$ such that $\Pi \cap \PP^{k_\a - 1} = \Pi_\a$. Note that the intersection remains unchanged however the coefficients $c_\gamma$, $\gamma \in \Delta (\a)$, are chosen. Consequently, the choice of $\Pi$ mentioned above is not unique. A few remarks are in order, and we summarize them below.

\begin{remark}\label{reduction}\normalfont
    \
    \begin{enumerate}
        \item[(a)] Our notation $\Omb$ for the $\a$-th Schubert variety, emphasizing its dependence on the basis $\B$, while letting $G(\ell, V_m)$ denote the Grassmannian independent of basis, might lead to some confusion. As mentioned earlier, we see that when we change the basis of $V_m$, it only induces a projective linear isomorphism of the Pl\"ucker space that fixes $G(\ell, V_m)$ as a set. It is understood, of course, that the homogeneous coordinates of the Pl\"ucker space get changed. However, for two bases $\B$ and $\B'$ of $V_m$, the sets $\Omega_\a (\ell, V_m, \B)$ and $\Omega_\a (\ell, V_m, \B')$ can be different as subsets. 
        \item[(b)] As before, let $\B =\{v_1, \dots, v_m\}$ be a fixed basis of $V_m$ and for each $i=1, \dots, m$, let $V_i$ denote the subspace spanned by $\{v_1, \dots, v_i\}$. It is easily seen that $\Omega_\a (\ell, V_m, \B) = \Omega_\a (\ell, V_{\a_\ell}, \B_{\a_\ell})$ where $\B_{\a_\ell} = \{v_1, \dots, v_{\a_\ell}\}$. In what follows, we shall use this identity. Indeed, this means we can view $\Omega_{\a} (\ell, V_m, \B)$ as a subset of $G(\ell, V_{\a_\ell})$. From now on, when we mention the Schubert subvariety $\Omega_\a (\ell, V_m, \B)$, we will always assume that $\a_\ell = m$ and that $V_{\a_\ell} = V_m$. 
    \end{enumerate}
\end{remark}

A non-zero element $z \in \bigwedge^{m- \ell} V_m$ is said to be \emph{decomposable} if there exist $w_1, \dots, w_{m-\ell} \in V_m$ such that $z = w_1 \wedge \cdots \wedge w_{m-\ell}$. A hyperplane of $\PP\left(\bigwedge^\ell V_m\right)$ is said to be \emph{decomposable} if it is given by a decomposable element of $\bigwedge^{m- \ell} V_m$. Thus a hyperplane $\Pi$ of ${\PP\left(\bigwedge^{\ell} V_m\right)}$ given by the equation $\displaystyle{\sum_{\a \in I(\ell, m)}} c_\a X_\a$ is decomposable if $\displaystyle{\sum_{\a \in I(\ell, m)}\epsilon(\a) c_\a v_{\a^{\mathsf{C}}}}$ is a decomposable element of $\bigwedge^{m - \ell} V_m$. 

\begin{remark}\label{rem1:dec}\normalfont
    For a non-zero element $z \in \bigwedge^{m - \ell} V_m$, we define $\V (z) = \{x \in V_m : z \wedge x = 0\}$. We have
    \begin{enumerate}
        \item[(a)] \cite[Section 4.1, Theorem 1.1]{M}  $z$ is decomposable if and only if $\dim \V (z) = m - \ell$. 
        \item[(b)] \cite[Section 4.1, Theorem 1.3]{M} if $\ell = 1$, then $z$ is decomposable. 
    \end{enumerate}
 \end{remark}

We now recall a suitable generalization of the notion of decomposable hyperplanes in the setting of Schubert subvarieties as given in \cite[Definition 2.1]{GS}.
\begin{definition}[Schubert Decomposability, \cite{GS}]\label{schdec}\normalfont
Suppose $\a \in I(\ell, m)$ and $\Omb$ be the $\a$-th Schubert subvariety of $G(\ell, V_m)$ with respect to the fixed basis $\{v_1, \dots, v_m\}$. 
We divide $\a$ into \emph{blocks of consecutive components} as follows:
$$(\a_1, \dots, \a_\ell) = (\a_1, \dots, \a_{p_1}, \a_{p_1 + 1}, \dots, \a_{p_2}, \dots, \a_{p_u +1}, \dots, \a_\ell),$$
where $u, p_0, \dots, p_{u+1}$ are unique integers determined by $\a$ satisfying
\begin{enumerate}
    \item[(a)] $p_0 := 0 < p_1 < \dots < p_u < p_{u+1} =:\ell$,
    \item[(b)] $\a_{p_i + 1} - \a_{p_i} \ge 2$ for $1 \le i \le u$, and
    \item[(c)] $\a_{p_i - j} = \a_{p_i} - j$ for $1 \le i \le u+1$ and $1 \le j <p_i - p_{i-1}$.
\end{enumerate}
This allows us to consider the \emph{refined partial flag} of subspaces associated with $\a$. That is, we now consider the partial flag:
$$0 \subset V_{\a_{p_1}} \subset \cdots \subset V_{\a_{p_u}} \subset V_m \ \ \ \text{where} \ \ V_{\a_i} = \text{span} \{v_1, \dots, v_{\a_i}\} \ \text{for} \ i=p_1, \dots, p_u.$$

An element $z \in \bigwedge^{m-\ell} V_m$ is said to be \emph{Schubert decomposable with respect to} $\Omb$ if $z$ is decomposable and $\dim (\V (z) \cap V_{\a_{p_i}}) = \a_{p_i} - p_i$ for all $i = 1, \dots, u$. A hyperplane $\Pi_\a$ of $\mathbb{P}^{k_{\a}-1}$ given by $F_\a$ is said to be a \emph{Schubert decomposable hyperplane} if there exists a Schubert decomposable element $F \in \bigwedge^{m-\ell} V_m$ such that $F_{\nabla (\a)} = F_\a$ and the hyperplane $\Pi$ defined by $F$ in $\PP^{k-1}$ restricts to $\Pi_\a$ in $\PP^{k_\a - 1}$.
\end{definition}

\section{Schubert varieties over finite fields: a decomposition into ``smaller" Schubert varieties, and a combinatorial upper bound}
From now on, we will denote by $\Fq$ a finite field with $q$ elements where $q$ is a fixed prime power. In this section, we will focus our attention on the study of Grassmannians and their Schubert subvarieties over $\Fq$. One of our main goals in this paper is to answer the following question:

\begin{question}\label{Q1}
Let $\ell$ and $m$ be positive integers satisfying $1 \le \ell \le m$. Let $\a \in I(\ell, m)$, $V_{m}$ be an $m$-dimensional vector space over $\Fq$ with a fixed ordered basis $\B$,  $G(\ell, V_m)$  the Grassmannian of all $\ell$ dimensional subspaces of $V_m$ and $\Omega_\a (\ell, V_m, \B)$ be the Schubert subvariety of $G(\ell, V_m)$ with respect to $\B$. 
\begin{enumerate}
    \item[(a)] Determine $e_\a (\ell, m) := \max \{ |\Pi \cap \Omega_{\a} (\ell, V_m, \B)| : \Pi \ \text{is a hyperplane of} \ \PP^{k_\a -1}\}$. 
    \item[(b)] Classify all the hyperplanes $\Pi$ of $\PP^{k_\a - 1}$ such that $|\Pi \cap \Omega_{\a} (\ell, V_m, \B)| = e_\a (\ell, m)$.
\end{enumerate}
\end{question}

Indeed, the quantity $e_\a (\ell, m)$ is independent of the choice of basis of $V_m$, and consequently, there is no scope for ambiguity with a basis-independent notation being used for the same. In the special case when $\Omb = G(\ell, V_m)$ the answer to Question \ref{Q1} is known due to Nogin \cite{N}. 
The answer to Question \ref{Q1} (a) is known thanks to \cite{X} and \cite{GS}. In fact, we have the following:
\begin{theorem}\cite[Theorem 2]{X}, \cite[Theorem 3.6]{GS}\label{thm}
    We have $e_\a (\ell, m) = |\Omb| - q^{\delta (\a)}$. 
\end{theorem}
As mentioned in the Introduction, the above theorem was proved in the special case when $\ell =2$ in \cite{C} and independently in \cite{GV}. We have the following conjecture \cite[Conjecture 5.6]{GS} which we will refer to as the \emph{minimum weight codewords conjecture} or \emph{MWCC} in  this article. 

\begin{conjecture}[MWCC, \cite{GS}]\label{conj} 
    If a hyperplane $\Pi$ satisfies $|\Pi \cap \Omega_{\a} (\ell, V_m, \B)| = e_\a (\ell, m)$, then $\Pi$ is Schubert decomposable. 
\end{conjecture}

\begin{remark}\label{where}\normalfont
The conjecture is known to be true in the following special cases: 
\begin{enumerate}
    \item[(a)] \cite[Theorem 5.5]{GS} If $\Pi$ is decomposable, that is, $\Pi$ extends to a decomposable hyperplane in $\PP(\bigwedge^\ell V_m)$;
    \item[(b)] \cite[Theorem 6.1 and Corollary 6.2]{GS} If $\a_1, \dots, \a_\ell$ are completely non-consecutive;
    \item[(c)]If $\a_1, \dots, \a_\ell$ are consecutive: in fact, from the discussion in Remark \ref{reduction}, the corresponding Schubert variety is a Grassmannian $G(\ell, V_m)$;
    \item[(d)] \cite[Corollary 6.3]{GS} In particular, if $\ell = 2$.
\end{enumerate}
\end{remark}
It is also known \cite[Theorem 5.3]{GS} that if $\Pi$ is Schubert decomposable, then $|\Pi \cap \Omega_{\a} (\ell, V_m, \B)| = e_\a (\ell, m)$. 

\subsection{Enumeration of Schubert varieties over finite fields}\label{subsec:en}
The first question one would encounter is the following: How many points are there on a Schubert variety $\Omb$ defined over a finite field $\Fq$? Indeed, if $\a = (m - \ell+1, \dots, m)$, then $\Omb = G(\ell, V_m)$ and its cardinality is given by the well-known \emph{Gaussian Binomial coefficient} ${m \brack\ell}_q$ defined as follows:
$${m \brack \ell}_q := \frac{(q^m -1) \cdots (q^m - q^{\ell -1})}{(q^\ell -1) \cdots (q^\ell - q^{\ell -1})}.$$
The Gaussian binomials are well-studied in mathematics, especially in combinatorics. It is beyond the scope of the article to discuss several nice properties of them, but we refer the readers to \cite{A} for a more leisurely reading on this topic. 
We now look at the general cases of Schubert varieties $\Omb$. For this, let us look at the equation \eqref{Sch}. From this, we already know the following:
\begin{equation}\label{schen}
    |\Omb| = \sum_{\beta \in \nabla (\a)} |C_\beta (\ell, V_m, \B)| = \sum_{\beta \in \nabla(\a)} q^{\delta(\beta)} \ \ \text{where} \ \ \delta(\beta) = \sum_{i=1}^\ell \beta_i - \frac{\ell (\ell + 1)}{2}.
\end{equation}
The last equality above follows from \eqref{Sch} and the well-known fact that $C_\beta (\ell, V_m, \B)$ is an affine space of dimension $\delta (\beta)$ over $\Fq$. This fact is easy to verify directly from looking at the matrices in right-row-reduced-echelon form that belong to $C_\beta (\ell, V_m, \B)$. It is difficult to trace back in the literature to find the first proof of the same, but we refer to \cite[Theorem 1]{GT} and the references therein.  To that end, 
 we now present a rather weak upper bound for the number of points on $\Omb$. We begin with a Lemma. 

\begin{lemma}\label{lem:ineq}
    For $\ell \ge 2$ and $\a \in I(\ell, m)$, we have 
    $$\sum_{\beta \in \nabla (\a)} q^{\beta_1 + \dots + \beta_\ell} < \left(\frac{q}{q-1}\right)^\ell q^{\a_1 + \dots + \a_\ell}.$$
\end{lemma}

\begin{proof}
    We prove this assertion by induction on $\ell$. We first look at the case when $\ell = 2$. Fix $\a = (\a_1, \a_2) \in I(2, m)$. We have
    \begin{align*}
        \sum_{(\beta_1, \beta_2) \in \nabla (\a)} q^{\beta_1 + \beta_2} &= \sum_{\beta_1 = 1}^{\a_1} \sum_{\beta_2 = \beta_1 + 1}^{\a_2} q^{\beta_1 + \beta_2} \\
        &= \sum_{\beta_1 = 1}^{\a_1} q^{\beta_1} \left(\sum_{\beta_2 = \beta_1 + 1}^{\a_2} q^{\beta_2}\right) \\
        &= \sum_{\beta_1 = 1}^{\a_1} q^{\beta_1} \left(q^{\beta_1 + 1} \frac{q^{\a_2 - \beta_1} -1}{q-1} \right) \\
        &< \frac{1}{q-1} \sum_{\beta_1 = 1}^{\a_1} q^{\beta_1} q^{\beta_1 + 1} q^{\a_2 - \beta_1} \\
        &= \frac{q^{\a_2 + 2}}{q-1} \sum_{\beta_1 = 1}^{\a_1} q^{\beta_1 - 1} \\
        &= \frac{q^{\a_2 + 2}}{q-1} \frac{q^{\a_1} -1}{q-1} < q^{\a_1 + \a_2} \left(\frac{q}{q-1}\right)^2.
    \end{align*}
Now assume that $\ell \ge 3$ and the assertion is true for every  $\a \in I(\ell - 1, m)$. Fix $\a = (\a_1, \dots, \a_\ell) \in I(\ell, m)$. We have
\begin{align*}
        \sum_{(\beta_1, \dots, \beta_\ell) \in \nabla (\a)} q^{\beta_1 + \dots + \beta_\ell}
        &= \sum_{(\beta_1, \dots, \beta_{\ell -1}) \in \nabla (\a_1, \dots, \a_{\ell - 1})}  q^{\beta_1 + \dots + \beta_{\ell -1} } \left( \sum_{\beta_\ell = \beta_{\ell -1} + 1}^{\a_\ell} q^{\beta_{\ell}} \right) \\
        &= \sum_{(\beta_1, \dots, \beta_{\ell -1}) \in \nabla (\a_1, \dots, \a_{\ell - 1})} q^{\beta_1 + \dots + \beta_{\ell -1} } q^{\beta_{\ell - 1} + 1} \left( \frac{q^{\a_\ell - \beta_{\ell - 1}} - 1}{q-1} \right)  \\
        &< \sum_{(\beta_1, \dots, \beta_{\ell -1}) \in \nabla (\a_1, \dots, \a_{\ell - 1})} q^{\beta_1 + \dots + \beta_{\ell -1} } q^{\beta_{\ell - 1} + 1} \left(\frac{q^{\a_\ell - \beta_{\ell - 1}}}{q-1}\right)  \\
        &= q^{\a_{\ell}} \frac{q}{q-1} \sum_{(\beta_1, \dots, \beta_{\ell -1}) \in \nabla (\a_1, \dots, \a_{\ell - 1})} q^{\beta_1 + \dots + \beta_{\ell -1} }  \\
        &< q^{\a_{\ell}} \frac{q}{q-1} q^{\a_1 + \dots + \a_{\ell - 1}} \left(\frac{q}{q-1}\right)^{\ell-1} = \left(\frac{q}{q-1}\right)^\ell q^{\a_1 + \dots + \a_\ell}.
    \end{align*}
    Here the strict inequality in the penultimate step follows from the induction hypothesis. This completes the proof. 
\end{proof}

\begin{proposition}\label{upp}
    For $\ell \ge 2$, we have $|\Omb| < \left(\frac{q}{q-1}\right)^\ell q^{\delta(\a)}$. 
\end{proposition}

\begin{proof}
    Follows directly from Lemma \ref{lem:ineq}, \eqref{schen}, and the definition $\delta (\a) = \a_1 + \dots + \a_\ell - \frac{\ell(\ell+1)}{2}$ for all $\a \in I(\ell, m)$.
\end{proof}

\begin{remark}\normalfont
\
\begin{enumerate}
\item[(a)] In order to completely understand the number of points on $\Omb$, one may write:
\begin{equation}\label{schen1}
|\Omega_\a (\ell, V_m, \B)| = \sum_{i=0}^{\delta (\a)} a_i q^i.
\end{equation}
The integers $a_i$ equal the number of Schubert cells of dimension $i$ in $\Omega_\a (\ell, V_m, \B)$. It is in general not difficult to determine these quantities. However, it is conceivable that determination of the coefficients $a_i$-s will give rise to a better bound for $|\Omb|$.
\item[(b)] A bound analogous to Proposition \ref{upp} for Gaussian binomials can be found in the literature. For example, it follows from \cite[Lemma 3.5]{NP} that ${m \brack \ell}_q < \frac{q^2}{q^2 - q - 1} q^{\ell (m-\ell)}$. Indeed, in the special case when the concerned Schubert subvariety is a Grassmannian, we see that the above bound is stronger in general, while they coincide with our bound when $q = 2$ and $\ell = 2$. We believe that the bounds we presented above can be improved and leave this for future investigations. 
    \end{enumerate} 
\end{remark}

\subsection{A combinatorial bound on Schubert subvarieties}
As noted above, it is not always easy to explicitly determine the coefficients $a_i$ as presented in equation \eqref{schen1}. Here, our goal is to bypass the determination of these quantities while still obtaining some effective combinatorial identities on the number of points on Schubert varieties, which are of independent interest and facilitate calculations. We stick to our basic definitions of $\ell, m, \a$ and work with a fixed ordered basis $\B$ of $V_m$.
First, we need to introduce some more notations. 

\begin{definition}\label{def1}\normalfont
    Suppose $\a = (\a_1, \dots, \a_\ell) \in I(\ell, m)$ and $\a_1, \dots, \a_\ell$ are not consecutive integers. We define:
    \begin{enumerate}
        \item[(a)] $k := \max  \{j : \a_{j+1} - \a_{j} \ge 2\}$. Since $\a_1, \dots, \a_\ell$ are not consecutive integers, we have $1 \le k \le \ell - 1$. 
        \item[(b)] $\a' := (\a_1', \dots, \a_\ell')$ where
        $$\a_i' = 
        \begin{cases}
            \a_i \ \ \ \ \ \ \ \ \ \text{if} \ \ \ i \le k \\
            \a_i - 1 \ \ \text{if} \ \ \ i> k. 
        \end{cases}$$
        \item[(c)] $\widecheck{\a} := (\a_1, \dots, \a_{\ell -1})$.  
        \item[(d)] For every $\beta \in I(\ell, m)$ and $1 \le k \le \ell$, define $\beta^{(k)} = (\beta_1, \dots, \beta_k) \in I(k, m)$.
        \item[(e)] For $\gamma = (\gamma_1, \dots, \gamma_k) \in I(k, m)$ with $\gamma \in \nabla (\a^{(k)})$, we define
        $$C_\gamma^{(k)} (\ell, V_m, \B) :=\left\{L \in \Omb : \text{the first $k$ pivots of} \ M_{\B}(L) \ \text{occur on the columns} \ \gamma_1, \dots, \gamma_k \right\}$$
        \end{enumerate}
\end{definition}

\begin{example}\normalfont
    Suppose $\a = (3,4,6,7) \in I(4, 7)$. In this case, $k =2$, $\a' = (3,4,5,6)$ and $\widecheck{\a} = (3,4,6)$. On the other hand, if $\a = (3, 4, 6, 8)$, then we have $k = 3$, $\a' = (3, 4, 6, 7)$ and $\widecheck{\a} = (3,4,6)$.
\end{example}

\begin{definition}\label{def2}\normalfont
    Let $\a, \ell, m, \a', \widecheck{\a}$ be as in Definition \ref{def1} and as before, let $\B = \{v_1, \dots, v_m\}$ be a fixed ordered basis of $V_m$. We define 
    \begin{enumerate}
        \item[(a)] $V_i = \text{span} \{v_1, \dots, v_i\}$ for $i = 1, \dots, m$. In particular, we have the $(m-1)$-dimensional subspace $V_{m-1}$ of $V_m$ that is spanned by $\{v_1, \dots, v_{m-1}\}$. 
        \item[(b)] $\B'=\{v_1, \dots, v_{m-1}\}$.
    \end{enumerate}
    In particular, when $\ell \le m-1$, the Schubert subvarieties $\Ombb$ and $\Ombh$ are well-defined. 
    \begin{enumerate}
        \item[(i)] Using the definition of $\a'$, we have
    $$\Ombb := \left\{L \in G(\ell, V_m) \middle\vert  \begin{array}{l} \dim L \cap V_{\a_i} \ge i  \ \ \text{if}  \ \ 1 \le i \le k \\ \dim L \cap V_{\a_i - 1} \ge i  \ \ \text{if} \ \  k < i \le \ell \end{array}\right\}.$$
    \item[(ii)] Similarly, with the definition of $\widecheck{\a}$, we have
    $$\Ombh := \{L \in G(\ell - 1, V_{m-1}) | \dim L \cap V_{\a_i} \ge i \ \text{for all} \ 1 \le i \le \ell -1\}.$$
    \end{enumerate}
    \end{definition}

It is evident that $\Ombb \subset \Omb$. Furthermore, $\Ombb$ admits its own cellular decomposition. In fact, we may write 
\begin{equation}\label{a'}
    \Ombb = \bigsqcup_{\beta \in \nabla (\a')} C_\beta (\ell, V_{m - 1}, \B').
\end{equation}
Let us note an interesting coincidence below. 

\begin{lemma}\label{lemcell}
    For each $\beta \in \nabla (\a')$, we have $C_\beta (\ell, V_{m-1}, \B') = C_\beta (\ell, V_m, \B)$. 
\end{lemma}

\begin{proof}

Let $\beta \in \nabla(\alpha')$. By definition,
\[
C_\beta(\ell, V_{m-1}, \B')
=
\left\{
L \in G(\ell, V_{m-1}) :
\dim(L \cap V_{\beta_i}) = i \text{ for all } i=1,\dots,\ell
\right\}.
\]
Since $V_{m-1} \subset V_m$ and $\B'=\{v_1,\dots,v_{m-1}\}\subset \B=\{v_1,\dots,v_m\}$, for any  $\beta\in\nabla(\alpha')$,  the flag
\[
0 \subset V_{\beta_1} \subset \cdots \subset V_{\beta_\ell}
\]
is identical whether regarded as a flag in $V_{m-1}$ or in $V_m$. The defining rank conditions of the Schubert cell $C_\beta$ in $\Ombb$ and $\Omb$ coincide. 
Consequently, a subspace $L \subset V_{m-1}$ satisfies the defining conditions of
$C_\beta(\ell, V_{m-1}, \B')$ if and only if it satisfies those of
$C_\beta(\ell, V_m, \B)$. Hence,
\[
C_\beta(\ell, V_{m-1}, \B') = C_\beta(\ell, V_m, \B),
\]
as required.
\end{proof}

Let us now introduce some more definitions. 
\begin{definition}\label{deflift}\normalfont
Let $\a, \ell, k$ be as in Definition \ref{def1}. Define a map $\phi : \nabla (\a') \to \nabla (\a)$ given by $$(\beta_1, \dots, \beta_k, \beta_{k+1}, \dots, \beta_\ell) \mapsto (\beta_1, \dots, \beta_k, \beta_{k+1}+1, \dots, \beta_\ell+1).$$
Evidently, the map $\phi$ is well-defined. 
The map $\phi$ induces a map
$$\Phi : \{C_\beta (\ell, V_{m-1}, \B') | \beta \in \nabla (\a')\} \to \{C_\beta (\ell, V_{m}, \B) | \beta \in \nabla (\a)\}$$ given by
\begin{equation}\label{Phi}
\Phi(C_\beta (\ell, V_{m-1}, \B')) = C_{\phi(\beta)} (\ell, V_m, \B).
\end{equation}
The Schubert cells of $\Omb$ that are not images of $\Phi$ are called the \emph{missing cells}. 
\end{definition}

Indeed, the definition of missing cells is not vacuous. One can easily check that the cell corresponding to $\beta = (1, \dots, \ell)$ is a missing cell. Moreover, there are some missing cells of very high dimensions. For example, the cell corresponding to $\a'' = (\a_1, \dots, \a_k, \a_{k} + 1, \a_{k+2}, \dots, \a_{\ell})$ is a missing cell with dimension 
$\delta (\a'') = \delta (\a) - (\a_{k+1} - \a_{k} - 1)$. In fact, since for every $\beta \in \nabla(\a')$, we have $\beta_k < \beta_{k+1}$, it follows that $\gamma \in I(\ell, m)$ with $\gamma_{k+1} - \gamma_k = 1$ corresponds to a missing cell.

\begin{lemma}\label{lb}
   Let $\a, \ell, k$ be as in Definition \ref{def1}.  For $\beta \in \nabla (\a')$, we have
   \begin{enumerate}
       \item[(a)] $|C_{\phi(\beta)} (\ell, V_m, \B)| = q^{\ell - k}  |C_{\beta} (\ell, V_{m-1}, \B')|$.
       \item[(b)] $q^{\a_{k+1} - \a_k - 1} (|\Omb| - q^{\ell - k}|\Ombb|) \ge q^{\delta (\a)}$. 
       \item[(c)] If  $k < \ell -1$, then $q^{\a_{k+1} - \a_k - 1} (|\Omb| - q^{\ell - k}|\Ombb|) \ge q^{\delta (\a)} + q^{\delta({\a'})}$.
   \end{enumerate}
 \end{lemma}   

\begin{proof}
    \
    \begin{enumerate}
        \item[(a)]  The first assertion follows since the left hand side equals $q^{\delta(\phi(\beta))}$, while the right hand side equals $q^{\delta (\beta) + \ell - k}$ and by definition $\delta (\phi(\beta)) = \delta(\beta) + \ell -k$. 
        \item[(b)] Let us denote by $I_1 = \{\gamma  \in \nabla (\a) : \gamma = \phi (\beta) \ \text{for some} \ \beta \in \nabla (\a')\}$ and $I_2 = \nabla (\a) \setminus I_1$. Note that for every $\gamma \in I_2$, the cell $C_\gamma (\ell, V_m, \B)$ is a missing cell, as defined above. We have
        \begin{align*}
            & \ \ |\Omb| - q^{\ell - k} |\Ombb| \\
            &= \sum_{\gamma \in \nabla (\a)} |C_{\gamma} (\ell, V_m, \B)| - q^{\ell - k} \sum_{\beta \in \nabla (\a')} |C_\beta(\ell, V_{m-1}, \B')| \\
            &=\sum_{\beta \in \nabla (\a')} (|C_{\phi (\beta)} (\ell, V_m, \B))| - q^{\ell - k} |C_\beta(\ell, V_{m-1}, \B')|)  + \sum_{\gamma \in I_2} |C_{\gamma} (\ell, V_m, \B)| \\
            &= \sum_{\gamma \in I_2} |C_{\gamma} (\ell, V_m, \B)|.
        \end{align*}
        As discussed above for any $\a = (\a_1, \dots, \a_\ell)$, with $k \le \ell -1$, the cell corresponding to the tuple $\tilde\a = (\a_1, \dots, \a_k, \a_k + 1, \a_{k+2}, \dots, \a_\ell)$ is a missing cell. Moreover, 
        $$q^{\a_{k+1} - \a_k - 1} q^{\delta (\tilde\a)} = q^{\delta (\a)}.$$
        In particular, if $k = \ell - 1$, then  $\tilde\a = (\a_1, \dots, \a_{\ell - 1}, \a_{\ell - 1} + 1)$ and the same equality as above holds. This proves that 
        $$q^{\a_{k+1} - \a_k - 1} (|\Omb| - q^{\ell - k}|\Ombb|) \ge q^{\delta (\a)}$$
        for every $\a \in I(\ell, m)$. 

        \item[(c)] Let us assume that $k < \ell - 1$. Note that the tuple $\bar{\a} = (\a_1, \dots, \a_k, \a_k + 1, \a_{k+2} - 1, \dots, \a_\ell -1)$ gives a missing cell. Moreover, 
        \begin{align*}
        \delta (\bar{\a}) + \a_{k+1} - \a_k - 1 &= \sum_{i =1}^\ell {\a_i} - (\ell - k -1) - \frac{\ell (\ell + 1)}{2} \\
       & > \sum_{i =1}^\ell {\a_i} - (\ell - k) - \frac{\ell (\ell + 1)}{2} =  {\delta(\a')}.
        \end{align*}
        Note that $\tilde{\a} = \bar{\a}$ is possible only if $k = \ell -1$, which is not in the purview of this case. 
    \end{enumerate}
    This completes the proof. 
\end{proof}

We state and prove the following bound. Undoubtedly, the bound looks unmotivated at this stage, but it will be central in the proof of our main theorem. 

\begin{theorem}\label{ineq:prel}
     Let $\a, \ell, k$ be as in Definition \ref{def1}. For $q > q_0 (\ell) := \frac{2^{\frac{1}{\ell - 1}}}{2^{\frac{1}{\ell - 1}} - 1}$, we have 
     $$q^{\a_{k+1} -\a_k - 1} (|\Omb| - q^{\ell - k}|\Ombb|) > q^{\delta (\a')} + q^{m - \ell} |\Omega_{\widecheck{\a}} (\ell - 1, V_{m-1},\B')| - q^{\delta (\a)}.$$
\end{theorem}

\begin{proof}
    We distinguish the proof into two cases. 

\begin{enumerate}
    \item[\bf{Case 1:}] \emph{Suppose that $\a = (1, \dots, \ell - 1, m)$.} From Lemma \ref{lb}, we have 
    $$q^{\a_{k+1} -\a_k - 1} (|\Omb| - q^{\ell - k}|\Ombb|) \ge q^{\delta (\a)} = q^{m - \ell}.$$
    On the other hand, we have $|\Ombh| = 1.$
    Combining all these, we see that the right-hand side of the desired inequality equals $q^{\delta (\a')}$. The assertion follows since $q^{\delta (\a)} > q^{\delta (\a')}$. 

    \item[\bf{Case 2:}] \emph{Suppose that $\a$ is not of the form $(1, \dots, \ell - 1, m)$}. If $k < \ell -1$, then Lemma \ref{lb} shows that the left-hand side of the desired inequality is greater than or equal to $q^{\delta(\a)} + q^{\delta (\a')}$. On the other hand Proposition \ref{upp} shows that, 
    \begin{align*}
    q^{m - \ell} |\Ombh|  + q^{\delta (\a')} - q^{\delta (\a)}  &< \left(\frac{q}{q-1}\right)^{\ell - 1} q^{\delta (\widecheck{\a})}q^{m-\ell}  + q^{\delta (\a')} - q^{\delta (\a)}. \\
    &=\left(\frac{q}{q-1}\right)^{\ell - 1}q^{\delta (\a)}  + q^{\delta (\a')} - q^{\delta (\a)}.\end{align*}
    In order to prove the assertion, it is enough to show that $$q^{\delta (\a)} + q^{\delta (\a')} > \left(\frac{q}{q-1}\right)^{\ell - 1} q^{\delta (\a)}  + q^{\delta (\a')} - q^{\delta (\a)},$$ which is equivalent to the inequality 
    $$\left(\frac{q}{q-1}\right)^{\ell - 1} < 2. $$
    The above inequality holds if and only if 
    $$\left(\frac{q}{q-1}\right) < 2^{\frac{1}{\ell - 1}} \iff \frac{1}{q-1} <  2^{\frac{1}{\ell - 1}} - 1\iff q > 1 + \frac{1}{2^{\frac{1}{\ell - 1}} - 1}  = \frac{2^{\frac{1}{\ell - 1}}}{2^{\frac{1}{\ell - 1}} - 1} = q_0 (\ell).$$
Now we assume that $k = \ell -1$. Write $\a = (\a_1, \dots, \a_{\ell - 1}, \a_\ell= m), \a' = (\a_1, \dots, \a_{\ell -1}, \a_{\ell} -1=m-1)$ and $\check{\a} = (\a_1, \dots, \a_{\ell -1})$ as per definition. 
The left hand side of the desired inequality is given by
\begin{align*}
A: &= q^{m - \a_{\ell - 1} - 1} \sum_{\beta \in I_2} q^{\delta(\beta)}= q^{m-\a_{\ell - 1} - 1} \left(\sum_{\check{\beta} \le \check{\a}}q^{\delta(\check{\beta}) + \beta_{\ell -1} + 1 - \ell}\right) = \left(\sum_{\check{\beta} \le \check{\a}}q^{\delta(\check{\beta}) + \beta_{\ell -1} - \a_{\ell - 1} + m - \ell}\right)
\end{align*}
On the other hand, analyzing the right hand side of the desired inequality, we have
$$B:=q^{m-\ell} \sum_{\check{\beta} \le \check{\a}} q^{\delta(\check{\beta})} - q^{\delta(\a)} + q^{\delta(\a) -1} = q^{m-\ell} \sum_{\check{\beta} \le \check{\a}} q^{\delta(\check{\beta})} -q^{\delta(\a)} \left(1 -\frac{1}{q}\right).$$

We need to show that $A - B > 0$. Simplifying the terms, we get
\begin{align*}
    A - B &= q^{m-\ell} \sum_{\check{\beta} \le \check{\a}}q^{\delta(\check{\beta})} \left(q^{\beta_{\ell - 1} - \a_{\ell -1}} - 1\right) + q^{m-\ell} q^{\delta(\check{\a})}\left(1-\frac{1}{q}\right) \\
    &= q^{m-\ell} \left(\sum_{\check{\beta} \le \check{\a}}q^{\delta(\check{\beta})} \left(q^{\beta_{\ell - 1} - \a_{\ell -1}} - 1\right) +q^{\delta(\check{\a})}\left(1-\frac{1}{q}\right) \right) \\
    &= q^{m-\ell + \delta(\check{\a})} \left(\sum_{\substack{ \check{\beta} \le \check{\a} \\ \beta_{\ell - 1} < \a_{\ell -1}}}q^{\delta(\check{\beta}) - \delta(\check{\a})} \left(q^{\beta_{\ell - 1} - \a_{\ell -1}} - 1\right) + \left(1-\frac{1}{q}\right) \right).
\end{align*}
It is thus enough to show that $$D = \left(\left(1-\frac{1}{q}\right) -\sum_{\substack{ \check{\beta} \le \check{\a} \\ \beta_{\ell - 1} < \a_{\ell -1}}}q^{\delta(\check{\beta}) - \delta(\check{\a})} \left(1 - q^{\beta_{\ell - 1} - \a_{\ell -1}} \right)  \right) > 0.$$
For each $\check{\beta} \le \check{\a}$, define $d_i (\check{\beta}) = \a_i - \beta_i$. Note that the map $\check{\beta} \mapsto (d_1 (\check{\beta}), \dots, d_{\ell -1}(\check{\beta})) : \nabla(\check{\a}) \to \mathbb{Z}_{\ge 0} \times \dots \times \mathbb{Z}_{\ge 0}$ is an injective map. 
We may thus write 
\begin{align*}
\sum_{\substack{ \check{\beta} \le \check{\a} \\ \beta_{\ell - 1} < \a_{\ell -1}}}q^{\delta(\check{\beta}) - \delta(\check{\a})} \left(1 - q^{\beta_{\ell - 1} - \a_{\ell -1}} \right) &= \sum_{\substack{ \check{\beta} \le \check{\a} \\ \beta_{\ell - 1} < \a_{\ell -1}}} q^{-(d_1 (\check{\beta}) + \dots + d_{\ell - 1}(\check{\beta}))}(1 - q^{-d_{\ell -1} (\check{\beta})}) \\
&\le \sum_{d_1 \ge 0} \dots\sum_{d_{\ell -2} \ge 0} \sum_{d_{\ell -1} \ge 1} q^{-(d_1  + \dots + d_{\ell - 1})}(1 - q^{-d_{\ell -1}}) \\
&< \sum_{d_1 \ge 0} \dots\sum_{d_{\ell -2} \ge 0} \sum_{d_{\ell -1} \ge 1} q^{-(d_1 + \dots + d_{\ell - 1})} \\
&=\left(\frac{q}{q-1}\right)^{\ell -2} \left(\frac{1}{q-1}\right).
\end{align*}
In order to complete the proof, we must therefore show that
$$\left(\frac{q}{q-1}\right)^{\ell -2} \left(\frac{1}{q-1}\right) \le \left(\frac{q-1}{q}\right),$$
which is equivalent to the inequality $$\left(\frac{q}{q-1}\right)^{\ell -1} \le q-1.$$
Since $q > q_0(\ell)$, it follows that the left hand side of the above inequality is bounded above by $2$. Since $q_0 (\ell)$ is a strictly increasing function of $\ell$ and $q_0(2) =2$, the condition $q > q_0 (\ell)$ thus implies $q \ge 3$, which establishes the inequality.
\end{enumerate}
\end{proof}
Before concluding the subsection, we remark that the function $q_0 (\ell)$ is increasing in $\ell$.  
\subsection{A decomposition of Schubert subvarieties and hyperplane sections} This subsection is motivated by \cite[Section 3]{DD}. We will continue with the notations introduced in the previous sections and subsections. We note that the Schubert subvariety $\Ombb$ is naturally contained in the Schubert subvariety $\Omb$. One can prove it directly using the definitions. On the other hand, this is a direct consequence of Lemma \ref{lemcell}. Independently, one can also see that $\Ombb = \Omb \cap G(\ell, V_{m-1}) \subseteq \Omb$. Lemma \ref{lemcell} implies that 
\begin{equation}\label{st1}
\Omb = \Ombb \bigsqcup \left(\bigsqcup_{\beta \in \Delta (\a') \cap \nabla(\a)} C_\beta (\ell, V_m, \B) \right).
\end{equation}
As in \cite[Section 3]{DD}, we decompose $\displaystyle\left(\bigsqcup_{\beta \in \Delta (\a') \cap \nabla(\a)} C_\beta (\ell, V_m, \B) \right)$ into \emph{strings} to facilitate computations. First, we note that
\begin{equation}\label{pivot}
    \Delta (\a') \cap \nabla(\a) = \{\beta \in \nabla (\a) : \beta_\ell = m\}.
\end{equation}
We recall that the condition $\beta_\ell = m$ mentioned above is a consequence of our running assumption that in our Schubert variety $\Omb$, the index $\a = (\a_1, \dots, \a_\ell)$ satisfies $\a_\ell = m$. 
Before delving into the technical details, for the sake of clarity, let us revisit our notations related to matrix representations from \eqref{em}. Recall that the bijection in \eqref{em} associates every element of $G(\ell, V_m)$ to a matrix in $\M(\ell, m)$ with respect to the fixed basis $\B = \{v_1, \dots, v_m\}$. Let us denote by $\M_{\a} (\ell, m)$ the image of this map restricted to $\Omb$ and $\M_{\a'} (\ell, m-1)$ that of $\Ombb$. In particular, we have the following description of $\M_\a (\ell, m)$\:
$$\M_\a (\ell, m) = \{M \in \M (\ell, m) | \ \text{the pivot in the $i$-th row occur on $\beta_i$th column with} \ \beta_i \le \a_i, \ \forall i\}.$$
The description of $\M_{\a'} (\ell, m-1)$ can be given similarly. An important remark on $\M_{\a'} (\ell, m-1)$ is in order. As per our continued assumptions, we are treating the elements of $\M_{\a'} (\ell, m-1)$ as $\ell \times (m-1)$ matrices. However, when we say that we view $\M_{\a'} (\ell, m-1)$ as a subset of $\M_\a (\ell, m)$, we use the natural inclusion $\M_{\a'} (\ell, m-1)$ into $\M_\a (\ell, m)$ by juxtaposing a zero column to the right of each element of $\M_{\a'} (\ell, m-1)$. Thus, we may regard the elements of $\M_{\a'}(\ell, m-1)$ as $\ell \times m$ matrices with entries in $\FF$ so that they can safely be regarded as elements of $\M_\a (\ell, m)$.  
Inspired by \cite[Section 3]{DD}, we will provide a decomposition of $\M_\a (\ell, m)$ into $\M_{\a'} (\ell, m-1)$ and  union of \emph{strings}. This decomposition is essentially obtained through \eqref{st1} above. 
Let us write 
$$T_\a (\ell, m) := \{M \in \M_\a(\ell, m) | \text{the  pivot of the last row of $M$ occurs on $m$-th column}\}.$$
One can see that, under the correspondence in \eqref{em}, the set $T_\a (\ell, m)$ corresponds to 
$$\F_\a (\ell, V_m, \B) := \Omb \setminus \Ombb = \bigsqcup_{\beta \in \Delta (\a') \cap \nabla (\a)} C_\beta (\ell, V_m, \B).$$
Indeed, the second equality is a consequence of \eqref{st1}. It is understood that a stratification of $T_\a(\ell, m)$ into strings gives rise to a decomposition of $\F_\a (\ell, V_m, \B)$ into smaller components.

Let $M \in T_{\a}(\ell, m)$.  We may write
\begin{equation}\label{cm}
    M = \begin{pmatrix}
        \widecheck{M} &&  {\bf 0} \\ 
          c_M && 1
    \end{pmatrix},
\end{equation}
where $\widecheck{M} \in \M_{\widecheck{\a}}(\ell - 1, m-1)$, the symbol ${\bf 0}$ denotes a column matrix of length $(\ell - 1)$ with all entries equal to ${0}$, whereas $c_M$ is a row vector of length $m-1$, and  $1$ is the multiplicative identity of $\FF$. Since $M$ is in right-row-reduced echelon form, so is $\widecheck{M}$, which justifies the inclusion $\widecheck{M} \in \M_{\widecheck{\a}}(\ell - 1, m-1)$. The pivot columns of $\widecheck{M}$ are determined by those of $M$, and the entries in $c_M$ on the pivot columns of $\widecheck{M}$ are zero. 
Let us denote $(M_{n_1}, \dots, M_{n_{m-\ell}})$ the entries in the non-pivotal columns of $c_M$, where $n_1 < \cdots < n_{m-\ell}$. This defines a map 
$$s : T_\a (\ell, m) \to \FF^{m-\ell}.$$
The map $s$ is clearly surjective. We claim that $s^{-1} (\nu)$ can be identified with $\M_{\widecheck{\a}} (\ell- 1, m-1)$ for any $\nu = (\nu_1, \dots, \nu_{m-\ell})  \in \FF^{m-\ell}$. 
Define 
\begin{equation}\label{bij2}
\phi_\nu : \M_{\widecheck{\a}} (\ell-1, m-1) \to s^{-1} (\nu) \ \ \ \text{by} \ \ \ M' \mapsto \begin{pmatrix}
        {M'} &&  0 \\ 
          c_{M', \nu} && 1 \end{pmatrix},\end{equation}
where $c_{M', \nu}$ is uniquely determined by the non-pivot positions of $M'$ and $\nu$. Clearly, the map $\phi_\nu$ is a bijection. Using the bijection $\M_{\widecheck{\a}} (\ell - 1, m-1) \longleftrightarrow \Ombh$, we note that $|s^{-1} (\nu)| = |\Ombh|$ for all $\nu \in \Fq^{m-\ell}$. Moreover, it follows trivially that $s^{-1} (\nu) \cap s^{-1} (\nu') = \emptyset$ whenever $\nu \neq \nu'$. Thus,   
\begin{equation}\label{t}
  \M_\a(\ell, m) = \M_{\a'} (\ell, m-1) \bigsqcup T_\a(\ell, m)  \ \ \ \text{where} \ \ \         T_\a(\ell, m)  = \bigsqcup_{\nu \in \FF^{m-\ell}} s^{-1} (\nu).  
\end{equation}
\noindent We call $s^{-1} (\nu)$, \emph{the $\nu$-th string of $\Omb$ with respect to $\Ombb$}.  
We look at a quick enumerative application of the above decomposition. 

\begin{proposition}\label{ineq:dec}
    Let $\FF = \Fq$. Then 
    \begin{enumerate}
        \item[(a)] $|\Omb| = |\Ombb| + q^{m-\ell} |\Ombh|$.
        \item[(b)] with $k$ as defined above and $q > q_0 (\ell)$, we have
        \begin{align*}
        q^{\a_{k+1} -\a_k - 1} (|\Omb| - & q^{\ell - k}|\Ombb|) \\ 
        &> q^{\delta (\a')} + (|\Omb| - |\Ombb|) - q^{\delta (\a)}.
        \end{align*}
    \end{enumerate}
\end{proposition}

\begin{proof}
    Part (a) follows at once from \eqref{t} and the discussion above, while part (b) is a consequence of part (a) and Theorem \ref{ineq:prel}.
\end{proof}

We would like to understand how certain hyperplanes behave with the above decomposition of $\Omb$. Recall the following inclusion of $\Omb \ \text{and} \  \Ombb$ in $G(\ell, V_m)$. 

\[
\begin{tikzcd}[column sep=large, row sep=large]
\Omega_{\alpha'}(\ell, V_{m-1}, \mathcal{B'})
  \arrow[r, hook]
  \arrow[d, hook]
&
\mathbb{P}^{k_{\alpha'}-1}
  \arrow[d, hook]
\\
\Omega_{\alpha}(\ell, V_m, \mathcal{B})
  \arrow[r, hook]
  \arrow[d, hook]
&
\mathbb{P}^{k_{\alpha}-1}
  \arrow[d, hook]
\\
G(\ell, V_m)
  \arrow[r, hook]
&
\mathbb{P}^{k-1}
\end{tikzcd}
\]

Let  $\Pi_\a$ be a hyperplane of $\PP^{k_\a - 1}$ given by a homogeneous polynomial $F_\a (X_\beta | \beta \in \nabla (\a))$. 
\begin{enumerate}
    \item[(I)] \emph{Note that $\Pi_\a$ does not contain $\Omb$}.
This follows from the non-degeneracy of $\Omb$ in $\PP^{k_\a - 1}$. Thus $F_\a \neq 0$. We write 
$$F_\a = F_{\a'} (X_\beta| \beta \in \nabla (\a') )+ F' (X_\beta | \beta \in \nabla (\a) \cap \Delta (\a')).$$
Then $\Pi_\a$ restricts to a hyperplane $\Pi_{\a'} = \Pi_\a \cap \PP^{k_{\a'} - 1}$ of $\PP^{k_{\a'} - 1}$ given by the vanishing of  $F_{\a'}$ in $\PP^{k_{\a'} - 1}$, provided $F_{\a'} \neq 0$. 
\item[(II)] \emph{Suppose that $\Pi_\a$ contains $\Ombb$}. The non-degeneracy of $\Ombb$ in $\PP^{k_{\a'} - 1}$ implies that the assumption is equivalent to the statement that $F_{\a'} = 0$. Consequently, the polynomial $F_\a$ is a linear combination of $X_\beta$ as $\beta$ varies in the set $\nabla (\a) \cap \Delta(\a')$. We note that for each such $\beta \in \nabla (\a) \cap \Delta(\a')$, we have $\beta_\ell = \a_\ell = m$.  
It also follows from \eqref{t} that 
\begin{equation}\label{t1}
    |\Pi_\a \cap \Omb| = |\Pi_\a \cap \Ombb| + \sum_{\nu \in \Fq^{m - \ell}} |\Pi_\a \cap s^{-1} (\nu)|.
\end{equation}
\noindent    Since $\Ombb \subseteq \Pi_\a$, we see that $|\Pi_\a \cap \Ombb| = |\Ombb|$. Consequently, the equality above translates to 
\begin{equation}\label{t1.5}
    |\Pi_\a \cap \Omb| = |\Ombb| + \sum_{\nu \in \Fq^{m - \ell}} |\Pi_\a \cap s^{-1} (\nu)|.
\end{equation}
    \end{enumerate}
We will, until the end of this section, assume that $\Pi_\a$ is a hyperplane of $\PP^{k_\a - 1}$ satisfying (II) above.     
 In order to compute $ |\Pi_\a \cap \Omb|$, it is now imperative to better understand the set $\Pi_\a \cap s^{-1} (\nu)$ for each $\nu \in \Fq^{m-\ell}$. For a fixed $\nu \in \Fq^{m - \ell}$, recall that set  $s^{-1} (\nu)$ is given by
$$\left \{\begin{pmatrix} \widecheck{M} && \bf{0} \\ c  && 1\end{pmatrix} : \widecheck{M} \in \M_{\widecheck{\a}} (\ell - 1, m-1) \right\},$$
which as per the above discussion is in one-one correspondence with $\M_{\widecheck{\a}} (\ell - 1, m-1)$.
We thus obtain the following bijections given by natural identifications:
\begin{equation}\label{bijn}
    s^{-1} (\nu)  \longleftrightarrow \M_{\widecheck{\a}} (\ell - 1, m-1) \longleftrightarrow \Ombh.
\end{equation}

Take $M \in s^{-1} (\nu)$. For $\beta \in \Delta (\a') \cap \nabla (\a)$, we see that $\beta_\ell = m$ and consequently, $X_{\beta} (M) = X_{\widecheck{\beta}, \nu} (\widecheck{M})$, where the latter is determined by taking the $(\ell - 1) \times (\ell - 1)$ minor of $M$ given by the columns numbered $\beta_1, \dots, \beta_{\ell - 1}$  and the first $\ell - 1$ rows. 
For $$F_\a = \displaystyle\sum_{\beta \in \Delta (\a') \cap \nabla (\a)} c_\beta X_\beta, \ \ \text{we write} \ \ F_{\widecheck{\a}, \nu} = \displaystyle\sum_{\beta \in \Delta (\a') \cap \nabla (\a)}c_\beta X_{\widecheck{\beta}, \nu}.$$ For each $M \in \Omb \setminus \Ombb$, we have 
$F_\a (M) = F_{\widecheck{\a}, \nu} (M)$. Before proceeding further, we remark that the evaluation is \emph{independent} of $\nu$. Namely, if $M, M' \in \Omb \setminus \Ombb$, with $M \in s^{-1} (\nu)$ and $M' \in s^{-1} (\mu)$ for some distinct $\nu, \mu \in \Fq^{m-\ell}$, but $\widecheck{M} = \widecheck{M'}$, then $F_{\a} (M) = F_{\a} (M')$. In particular, we have
$$|\{M \in s^{-1} (\nu) : F_{\a} (M) = 0\}| = |\{M \in s^{-1} (\mu) : F_{\a} (M) = 0\}| \ \ \text{for} \ \ \nu, \mu \in \Fq^{m-\ell}.$$
Combining this with \eqref{t1.5}, we obtain
\begin{equation}\label{t2}
    |\Pi_\a \cap \Omb| = |\Ombb| + q^{m-\ell} |\Pi_\a \cap s^{-1} (\nu)| 
\end{equation}

Thus we may fix $\nu \in \Fq^{m-\ell}$ and focus on computing $|\Pi_\a \cap s^{-1} (\nu)|$. 
We have already identified $s^{-1} (\nu)$ with the Schubert subvariety $\Ombh$. Let us ponder on the Schubert subvariety $\Ombh$ a little more before proceeding further. Indeed, $\Ombh \subset G(\ell - 1, V_{m-1})$. Now $G(\ell - 1, V_{m-1})$ is embedded inside the projective space 
$\PP \left( \displaystyle{\bigwedge^{\ell-1}} V_{m-1}\right)$ which is a projective space of dimension $\widecheck{k} - 1$, where $\widecheck{k} = {m - 1 \choose \ell - 1}$. Moreover, as usual, the Schubert subvariety $\Ombh$ is obtained by taking an appropriate linear section of $G(\ell - 1, V_{m-1})$ given by the polynomials $X_{\widecheck{\beta}}$ where $\widecheck{\beta} \in I(\ell - 1, m-1)$ with $\widecheck{\beta} \not\leq \widecheck{\a}$. 
We see that these $\widecheck{\beta} = (\widecheck{\beta}_1, \dots, \widecheck{\beta}_{\ell - 1})$-s get naturally identified with $\beta \in \Delta (\a) \subset I(\ell, m)$ with $\beta_i = \widecheck{\beta}_i$ for $i = 1, \dots, \ell - 1$ and $\beta_{\ell} = m$. 
In fact, if we look at the identification $X_{\widecheck{\beta}}$ as an element of $\bigwedge^{m - \ell} V_{m-1} \subset \bigwedge^{m-\ell} V_m$, then it is actually equal to $X_{\beta}$.  
At any rate, we may assume that $\Ombh$ is embedded in $\PP^{k_{\widecheck{\a}} - 1}$, where the latter is the linear subspace of $\PP^{\widecheck{k} - 1}$ mentioned above. Thus, under the identification of $s^{-1} (\nu)$ with $\Ombh$ described in \eqref{bijn}, we see that the set $s^{-1} (\nu) \cap \Pi_\a$ gets identified with the set $\Ombh \cap \Pi_{\widecheck{\a}}$ where $\Pi_{\widecheck{\a}}$ is the hyperplane of $\PP^{k_{\widecheck{\a}} - 1}$.  

To summarize, let $\Pi_\a$ be a hyperplane of $\PP^{k_\a - 1}$ satisfying (I) and (II) above. Then $\Pi_\a$ is given by a linear homogeneous polynomial 
$$F_\a = \sum_{{\beta \in \Delta (\a') \cap \nabla (\a) }} c_\beta X_\beta.$$ 
Then $\Pi_\a$ restricts to the hyperplane $\Pi_{\widecheck{\a}}$ given by 
$$F_{\widecheck{\a}} = \sum_{{\beta \in \Delta (\a') \cap \nabla (\a)}} c_\beta X_{\widecheck{\beta}}.$$ 

\begin{remark}\label{remnotcontain}\normalfont
    We readily see that a hyperplane satisfying (I) and (II) above cannot contain $s^{-1}(\nu)$. Otherwise, it follows from \eqref{bijn}, \eqref{t2}, and Proposition \ref{ineq:dec} that $|\Pi_\a \cap \Omb| = |\Omb|$, which implies that $\Omb \subset \Pi_\a$, contradicting (I). In particular, we have $|\Pi_\a \cap s^{-1} (\nu)| < |s^{-1} (\nu)|.$
\end{remark}

Let us now introduce some notations for ease of reference.

\begin{definition}\normalfont
    Let $\a, \a',\widecheck{\a}, \ell, m$ be as defined above. With $\FF = \Fq$, we define
    \begin{enumerate}
        \item[(a)] $e_\a (\ell, m) := \max \{|\Omb \cap \Pi_\a| : \Pi_\a \ \text{is a hyperplane of} \ \PP^{k_{\a} - 1}\}$.
        \item[(b)] $e_{\a'} (\ell, m-1) := \max \{|\Ombb \cap \Pi_{\a'}| : \Pi_{\a'} \ \text{is a hyperplane of} \ \PP^{k_{\a'} - 1}\}$.
        \item[(c)] $e_{\widecheck{\a}} (\ell-1, m-1) := \max \{|\Ombh \cap \Pi_{\widecheck{\a}}| : \Pi_{\widecheck{\a}} \ \text{is a hyperplane of} \ \PP^{k_{\widecheck{\a}} - 1}\}$.
    \end{enumerate}
\end{definition}

We conclude this section with the following Proposition summarizing the discussion in the current section which can be directly used in the proof of our main theorem. 

\begin{proposition}\label{case1}
    Let $\FF = \Fq$. 
    If $\Pi_\a$ is a hyperplane of $\PP^{k_\a - 1}$ such that $\Pi_\a$ contains $\Ombb$, then 
    $$|\Pi_\a \cap \Omb| \le |\Ombb| + q^{m - \ell} e_{\widecheck{\a}}(\ell - 1, m-1).$$
\end{proposition}

\begin{proof}
From \eqref{t2}, we know that  $|\Pi_\a \cap \Omb| = |\Pi_\a \cap \Ombb| + q^{m-\ell} |\Pi_\a \cap s^{-1} (\nu)|$. Moreover, as mentioned in Remark \ref{remnotcontain}, the hyperplane $\Pi_\a$ does not contain $s^{-1} (\nu)$. Thus $|\Pi_\a \cap s^{-1} (\nu)| = |\Pi_{\widecheck{\a}} \cap \Ombh| \le e_{\widecheck{\a}}(\ell - 1, m-1).$ The assertion now follows trivially. 
\end{proof}

\section{Main Theorem}
Throughout this section, we will assume that $\Fq$ is a finite field with $q$ elements. All the notations $V_m, \a, \a', \widecheck{\a}, \Omb, \Ombb, \ \text{and} \ \Ombh$ are as discussed in the previous sections. Before presenting our first result, which is a Schubert-analog of \cite[Lemma 4.1]{DD}, we define a special type of containment of Schubert subvarieties. 

\begin{definition}\normalfont
Let $\ell \le m - 1$.
 Let $\B$ be as above. Let $W_{m-1}$ be an $(m-1)$ dimensional subspace of $V_m$ with an ordered basis $\B_1 = \{w_1, \dots, w_{m-1}\}$ extending $\{v_1, \dots, v_{\a_k}\}$. As usual, we denote by $W_i = Span \{w_1, \dots, w_i\}$ for $i =1, \dots, m-1$.  For $\a' = (\a_1, \dots, \a_k, \a_{k+1} - 1, \dots, \a_\ell - 1)$ as defined above, 
we say that $\Omega_{\a'} (\ell, W_{m-1}, \B_1)$ is \emph{strictly contained in} $\Omb$ if $V_{\a_i} = W_{\a_i} \ \text{for} \ i = 1, \dots, k$ and $V_{\a_k} \subset W_{m-1} \subset V_m$. 
In this case, we write $\Omega_{\a'} (\ell, W_{m-1}, \B_1) \le \Omb$ and define
$$\mathfrak{S}_\a (\ell, V_m, \B) := \{\Omega_{\a'} (\ell, W_{m-1}, \B_1) \le \Omb\}.$$
We note (see Remark \ref{rem:B1indep} below) that $\Omega_{\a'} (\ell, W_{m-1}, \B_1)$ depends only on $W_{m-1}$ and not on the basis $\B_1$ extending $\{v_1, \dots, v_{\a_k}\}$.
\end{definition}

\begin{remark}\label{rem:B1indep}\normalfont
\
\begin{enumerate}
\item The set $\Omega_{\a'}(\ell, W_{m-1}, \B_1)$ depends only on $W_{m-1}$, and not on 
the choice of ordered basis $\B_1$ extending $\{v_1, \dots, v_{\a_k}\}$. 
Indeed, by Definition \ref{def2}, $\Omega_{\a'}(\ell, W_{m-1}, \B_1)$ consists 
of those $L \in G(\ell, W_{m-1})$ satisfying
$$\dim (L \cap V_{\a_i}) \ge i \quad (1 \le i \le k)
\qquad \text{and} \qquad 
\dim (L \cap W_{\a_i - 1}) \ge i \quad (k < i \le \ell).$$
We claim that the conditions in the second family hold automatically for every 
$L \in G(\ell, W_{m-1})$, and hence carry no dependence on $\B_1$. Since 
$\a_{k+1}, \dots, \a_\ell$ are consecutive with $\a_\ell = m$, we have 
$\a_i - 1 = (m-1) - (\ell - i)$ for $k < i \le \ell$, so for any flag
$V_{\a_1} \subset \dots \subset V_{\a_k} \subset {W_{{\a_k}+1}} \subset \dots \subset W_{m-1}$, we have
$$\dim (L \cap W_{\a_i - 1}) \ \ge \ \dim L + \dim W_{\a_i -1} - \dim W_{m-1} 
\ = \ \ell + \bigl((m-1) - (\ell - i)\bigr) - (m-1) \ = \ i,$$
as claimed. Therefore
$$\Omega_{\a'}(\ell, W_{m-1}, \B_1) = \left\{ L \in G(\ell, W_{m-1}) : 
\dim (L \cap V_{\a_i}) \ge i \ \text{ for } \ i = 1, \dots, k \right\},$$
which depends only on $W_{m-1}$ and the fixed subspaces $V_{\a_1}, \dots, 
V_{\a_k}$, as required.
\item Conversely, distinct admissible subspaces give rise to distinct Schubert subvarieties. 
Indeed, let $W_{m-1} \ne W'_{m-1}$ be two subspaces of $V_m$ each containing $V_{\a_k}$ as a subspace, and choose 
$w \in W_{m-1} \setminus W'_{m-1}$. Put $U = \mathrm{Span}\{v_1, \dots, v_k\}$. Since $k \le \a_k$, we have $U \subseteq V_{\a_k} \subseteq W_{m-1}$. Moreover, since $U \subseteq 
V_{\a_k} \subseteq W'_{m-1}$, we have $w \notin U$. As $\dim W_{m-1} = m-1 \ge \ell$, we may extend 
$U + \mathrm{Span}\{w\}$ to an $\ell$-dimensional subspace $L \subseteq W_{m-1}$. For 
$i = 1, \dots, k$ we have $\mathrm{Span}\{v_1, \dots, v_i\} \subseteq L \cap V_{\a_i}$, 
since $\a_i \ge i$. Thus $\dim (L \cap V_{\a_i}) \ge i$. Consequently $L \in \Omega_{\a'}(\ell, W_{m-1}, \B_1)$. On the other hand, 
$w \in L$ and $w \notin W'_{m-1}$, so $L \not\subseteq W'_{m-1}$ and consequently 
$L \notin \Omega_{\a'}(\ell, W'_{m-1}, \B_1')$ for any basis $\B_1'$ of $W_{m-1}'$. Hence 
$\Omega_{\a'}(\ell, W_{m-1}, \B_1) \ne \Omega_{\a'}(\ell, W'_{m-1}, \B_1')$. 
\end{enumerate}
\end{remark}

\begin{remark}\label{salpha}\normalfont
    Before proceeding, it is worth understanding the above definition and its consequences a bit more carefully. Let us take $\Omega_{\a'} (\ell, W_{m-1}, \B_1) \in \mathfrak{S}_\a (\ell, V_m, \B)$. We write $\B_1 = \{w_1, \dots, w_{m-1}\}$ which serves as the fixed ordered basis of $W_{m-1}$. Let us extend it to an ordered basis $\tilde{\B} = \{w_1, \dots, w_{m-1}, w_m\}$ of $V_m$. Let us consider the Schubert subvariety $\Omega_{\a} (\ell, V_m, \tilde{\B})$. This is given by the flag 
        $$V_{\a_1} \subset \cdots \subset V_{\a_k} \subset W_{m-(\ell - k) + 1} \subset \cdots \subset W_{m-1} \subset V_m.$$
          Note that, for any $L \in \Omb$, we have the Schubert conditions:
        $\dim (L \cap V_{\a_i}) \ge i$ for $i = 1, \dots, \ell$. 
        Furthermore, for each $j$ with $k < j \le \ell$, one has $\a_j = m - \ell + j$, 
        so that for any subspace $U \subseteq V_m$ with $\dim U = m - \ell + j$ and any 
        $L \in G(\ell, V_m)$, we have
        $$\dim (L \cap U) \ \ge \ \dim L + \dim U - \dim V_m 
        \ = \ \ell + (m - \ell + j) - m \ = \ j.$$
        Applying this with $U = V_{\a_j}$, and with $U = W_{\a_j}$ for $k < j \le \ell - 1$, 
        we see that the Schubert conditions at all the indices $j$ with $k < j \le \ell$ 
        hold automatically for both the flags; here we note that the condition at 
        $j = \ell$ reads $\dim (L \cap V_m) \ge \ell$, which is trivially satisfied. 
        Since the two flags agree at the positions $i = 1, \dots, k$, namely 
        $V_{\a_1} \subset \cdots \subset V_{\a_k}$, the Schubert subvarieties 
        $\Omb$ and $\Omega_{\a} (\ell, V_m, \tilde{\B})$ are defined by the same 
        conditions, we have $\Omb = \Omega_{\a} (\ell, V_m, \tilde{\B})$. 
        Moreover, as per Definition \ref{def2} (b), we have $\tilde{\B}'= \B_1$.
    \begin{enumerate}
        \item[(a)]  If a hyperplane $\Pi_\a$ of $\PP^{k_\a - 1}$ contains $\Omega_{\a'} (\ell, W_{m-1}, \B_1)$ for some $\Omega_{\a'} (\ell, W_{m-1}, \B_1) \in \mathfrak{S}_\a (\ell, V_m, \B)$, then we may assume, after a change of coordinates (induced by the change of bases from $\B$ to $\tilde{\B}$) that $\Pi_\a$ contains $\Ombb$. By doing this, we can use the tools we have developed above, as in Proposition \ref{case1}, among others.
        \item[(b)] Next, suppose $\Pi_\a$ is a hyperplane of $\PP^{k_\a - 1}$ such that $\Pi_\a$ does not contain $\Omega_{\a'} (\ell, W_{m-1}, \B_1)$. Again, with respect to the new coordinate system on $\PP^{k -1}$ and in particular $\PP^{k_\a - 1}$, we may assume that $\Omega_{\a'} (\ell, W_{m-1}, \B_1) = \PP^{k_{\a'} - 1} \cap \Omb$. We see that $\Pi_\a \cap \PP^{k_{\a'} - 1}$ is a hyperplane of $\PP^{k_{\a'} - 1}$ and we have
        $|\Omega_{\a'} (\ell, W_{m-1}, \B_1) \cap \Pi_\a| \le e_{\a'}(\ell, m-1)$. 
     \end{enumerate}
\end{remark}

\noindent We begin with the following proposition concerning the cardinalities of $\mathfrak{S}_\a (\ell, V_m, \B)$ and some of its special subsets.

\begin{proposition}\label{prop1}
   With the notations as above, we have 
   \begin{equation}\label{eq}
   |\mathfrak{S}_\a (\ell, V_m, \B)| = \frac{q^{m - \a_k} - 1}{q-1}.
   \end{equation}
   Moreover, for $L \in \Omb$, we have
    \begin{equation}\label{ineq}|\{\Omega_{\a'} (\ell, W_{m-1}, \B_1) \in \mathfrak{S}_\a (\ell, V_m, \B) : L \in \Omega_{\a'} (\ell, W_{m-1}, \B_1)\}| \ge \frac{q^{(m - \a_k) - (\ell - k)} - 1}{q-1}.\end{equation}
\end{proposition}

\begin{proof}
Note that $\Omega_{\a'} (\ell, W_{m-1}, \B_1) \le \Omb$ if and only if 
$V_{\a_i} = W_{\a_i}$ for all $i = 1, \dots, k$ and 
$V_{\a_k} \subset W_{m-1} \subset V_m$. Moreover, by Remark \ref{rem:B1indep} (1), the set $\Omega_{\a'}(\ell, W_{m-1}, \B_1)$ depends only on $W_{m-1}$, and by Remark \ref{rem:B1indep} (2), pairwise distinct subspaces $W_{m-1}$ give pairwise distinct Schubert subvarieties. Thus
$$|\mathfrak{S}_\a (\ell, V_m, \B)| = |\{W_{m-1} : V_{\a_k} \subset W_{m-1} 
\subset V_m, \ \dim W_{m-1} = m-1\}| = \frac{q^{m - \a_k} - 1}{q-1}.$$
This proves \eqref{eq}. 

Using a similar argument, we see that
\begin{align*}
&|\{\Omega_{\a'} (\ell, W_{m-1}, \B_1) : L \in \Omega_{\a'} (\ell, W_{m-1}, 
\B_1) \le \Omb\}| \\
&= |\{W_{m-1} : L + V_{\a_k} \subseteq W_{m-1} \subset V_m, \ 
\dim W_{m-1} = m-1\}| \\
&= \frac{q^{m - \dim(L + V_{\a_k})} - 1}{q - 1}.
\end{align*}
Since $L \in \Omb$, the Schubert condition gives 
$\dim (L \cap V_{\a_k}) \ge k$. Consequently,
$$\dim (L + V_{\a_k}) = \dim L + \dim V_{\a_k} - \dim (L \cap V_{\a_k}) 
\le \ell + \a_k - k.$$
It follows that
$$\frac{q^{m - \dim(L + V_{\a_k})} - 1}{q - 1} \ge 
\frac{q^{(m - \a_k) - (\ell - k)} - 1}{q - 1}.$$
This completes the proof.
\end{proof}

\begin{lemma}\label{dc}
    Let $\Pi_\a$ be a hyperplane of $\PP^{k_\a - 1}$. Then 
    $$|\Pi_\a \cap \Omb| \le \left(\frac{q^{m-\a_k} - 1}{q^{(m - \a_k) - (\ell - k)} - 1}\right) M_\a,$$ where $M_\a= \max\{|\Pi_\a \cap \Omega_{\a'} (\ell, W_{m-1}, \B_1)| : \Omega_{\a'} (\ell, W_{m-1}, \B_1) \in \mathfrak{S}_\a (\ell, V_m, \B)\}.$
\end{lemma}

\begin{proof}
    We define the following incidence set $\mathfrak{I}$ and count it in two ways:
    $$\mathfrak{I} := \left\{\left(L, \Omega_{\a'} (\ell, W_{m-1}, \B_1) \right) \in \Omb \times \mathfrak{S}_\a (\ell, V_m, \B)  :   \ L \in \Pi_\a \cap \Omega_{\a'} (\ell, W_{m-1}, \B_1)  \right\}.$$
    On one hand, we have 
    \begin{align*}
    |\mathfrak{I}| &= \sum_{L \in \Pi_\a \cap \Omb} \#\{\Omega_{\a'} (\ell, W_{m-1}, \B_1) | L \in \Omega_{\a'} (\ell, W_{m-1}, \B_1) \le \Omb\} \\
&\ge \sum_{L \in \Pi_\a \cap \Omb} \frac{q^{(m - \a_k) - (\ell - k)} - 1}{q-1} \\
&= |\Pi_\a \cap \Omb| \frac{q^{(m - \a_k) - (\ell - k)} - 1}{q-1}.
    \end{align*}
    The inequality above is a consequence of \eqref{ineq}. On the other hand, we have
    \begin{align*}
    |\mathfrak{I}| &= \sum_{\Omega_{\a'} (\ell, W_{m-1}, \B_1)\le \Omb} \#\{L \in \Pi_\a \cap \Omega_{\a'} (\ell, W_{m-1}, \B_1)\} \\
&\le \sum_{\Omega_{\a'} (\ell, W_{m-1}, \B_1) \le \Omb} M_\a 
=|\{\Omega_{\a'} (\ell, W_{m-1}, \B_1) \le \Omb\}| M_\a = \frac{q^{m - \a_k} - 1}{q-1} M_\a.
    \end{align*}
    The last equality follows from \eqref{eq}. 
We have thus obtained 
$$ |\Pi_\a \cap \Omb| \frac{q^{(m - \a_k) - (\ell - k)} - 1}{q-1} \le |\mathfrak{I}| \le \frac{q^{m - \a_k} - 1}{q-1} M_\a.$$
The assertion now follows trivially. 
\end{proof}

\noindent Let us now include an auxiliary Lemma that will be useful in the computations. 

\begin{lemma}\label{aux}
    For $q > q_0 (\ell)= \frac{2^{\frac{1}{\ell - 1}}}{2^{\frac{1}{\ell - 1}} - 1}$, we have
    $$\frac{q^{m - \a_k} - 1}{q^{(m- \a_k) - (\ell - k)} - 1}  \left(|\Ombb| - q^{\delta (\a')}\right) < |\Omb| - q^{\delta (\a)}.$$
\end{lemma}

\begin{proof}
    Note that the inequality in the assertion is equivalent to the inequality:
    \begin{align*}
        q^{(m - \a_k) - (\ell - k)} (|\Omb| - & q^{\ell - k} |\Ombb|)\\ 
        & > (|\Omb| - |\Ombb|) - q^{\delta (\a)} + q^{\delta(\a')} . 
    \end{align*}
   It follows from the definition of $k$ that $m = \a_\ell = \a_{k+1}+ (\ell - k - 1)$. Consequently, $(m - \a_k) - (\ell - k)  = \a_{k+1} - \a_k - 1$. The assertion now follows from Proposition \ref{ineq:dec} (b).
\end{proof}

\begin{remark}\label{trivialcases}\normalfont
    Before proving our main theorem, let us warm up by considering a few trivial cases. As before, let $1 \le \ell < m$, $V_m$ a vector space of dimension $m$ over $\Fq$ with an ordered basis $\B$ and $\a \in I(\ell, m)$. Our goal will be to prove the following results, when $q > q_0 (\ell)$: 
    \begin{enumerate}
    \item[(a)] $e_\a (\ell, m) = |\Omb| - q^{\delta(\a)}$.
    \item[(b)] If $\Pi_\a$ is a hyperplane such that $|\Pi_\a \cap \Omb| = |\Omb| - q^{\delta(\a)}$, then $\Pi_\a$ is Schubert decomposable. 
    \end{enumerate}
    As we have mentioned in Remark \ref{where} (a), it was proved by Ghorpade and Singh \cite{GS} that if $\Pi_\a$ is a hyperplane that satisfies $|\Pi_\a \cap \Omb| = |\Omb| - q^{\delta(\a)}$ and $\Pi_\a$ is decomposable, that is $\Pi_\a$ is given by a decomposable hyperplane of $\PP^{k-1}$, then $\Pi_\a$ is Schubert decomposable. Thus, while proving our main theorem, we will only restrict ourselves to proving that a hyperplane attaining the above bound is decomposable. Also, we will settle a few extreme cases completely here, so that we can ignore them in our main proof. 
    \begin{enumerate}
        \item[(0)] Suppose $\ell = 1$. Then $\Omega_\a (1, V_m, \B)$ is the linear  subspace $\PP (V_{\a_1})$ of $\PP^{m-1}$, so $|\Omega_\a (1, V_m, \B)| = 1 + \cdots + q^{\a_1 - 1}$, and any hyperplane not containing it meets it in 
        $1 + \cdots + q^{\a_1 - 2}$ points; thus $e_\a (1,m) = |\Omega_\a (1, V_m, \B)| - q^{\delta (\a)}$, since $\delta (\a) = \a_1 - 1$. Moreover, every element of $\bigwedge^{m-1} V_m$ is decomposable by Remark \ref{rem1:dec} (b), so both 
        assertions hold trivially.
        
        \item[(I)] Suppose $\a = (\a_1, \dots, \a_\ell)$ are consecutive. In this case, $\Omb$ is the Grassmannian $G(\ell, V_m)$ and the result is known in this case due to Nogin \cite{N}. Moreover, a method similar to this article was used in \cite{DD} to obtain an alternative proof of this result. 
        
        \item[(II)] Let us now consider the case when $\ell = m - 1$. In this case $G(\ell, V_{m})$ is the projective space $\PP^{m-1}$. We write $I(\ell, m) = \{\Gamma_1, \dots, \Gamma_{m}\}$, where $\Gamma_i = (1, 2, \dots, i-1, \widecheck{i}, i+1, \dots, m)$ for $i = 1, \dots, m$. That is, each $\Gamma_i$ is obtained by writing down the string of consecutive integers $1, \dots, m$ and dropping the integer $i$. Also, due to our standing convention, we do not need to consider the Schubert subvariety corresponding to $\Gamma_m = (1, \dots, m-1)$.  Let us now proceed to prove part (a) above and fix $\Gamma_i$. Note that, with respect to a fixed basis $\B$ of $V_m$, $\Omega_{\Gamma_i} (m-1, V_m, \B)$ is a linear subspace of $\PP^{m-1}$ of codimension $(i-1)$. Consequently, $|\Omega_{\Gamma_i} (m-1, V_m, \B)| = 1 + \dots + q^{m-i}$. Clearly, a hyperplane of $\PP^{m-1}$, that does not contain $\Omega_{\Gamma_i} (m-1, V_m, \B)$ intersects it at most at $1 + \dots + q^{m-i - 1}$ many points. This proves part (a). Moreover, part (b) is trivially true since any hyperplane of $\PP\left(\bigwedge^{m-1} V_m\right)$ is decomposable. 
    \end{enumerate} 
\end{remark}

\begin{theorem}\label{main}
    Let $q > q_0(\ell)= \frac{2^{\frac{1}{\ell - 1}}}{2^{\frac{1}{\ell - 1}} - 1}
    $ and $\ell,m$ be positive integers satisfying $2 \le \ell < m$. For $\a = (\a_1, \dots, \a_\ell) \in I(\ell, m)$ and any hyperplane $\Pi_\a$ in $\PP^{k_\a - 1}$, we have 
    $$|\Pi_\a \cap \Omega_\a(\ell, V_m, \B)| \le e_\a (\ell, m) := |\Omega_\a(\ell, V_m, \B)| - q^{\delta(\a)}.$$
    Moreover, the bound is attained if and only if $\Pi_\a$ is Schubert decomposable. 
\end{theorem}

\begin{remark}\normalfont
We note that the equality $e_\alpha(\ell,m) = |\Omb| - q^{\delta(\alpha)}$ is known from \cite{GS}. Our proof gives an independent verification of this fact for $q > q_0(\ell)$. In \cite{GS}, the authors also showed that if a hyperplane is Schubert decomposable, then it attains the bound $e_\a (\ell, m)$. The main new contribution of our proof is the converse of this fact, that is,  
only Schubert decomposable hyperplanes attain the bound. 
\end{remark}

\begin{proof}[Proof of Theorem \ref{main}]
   In view of Remark \ref{reduction} (b), we assume throughout that $\a_\ell = m$. We prove the assertions by induction on $m$. Note that the hypothesis mandates $m \ge 3$. When $m=3$, we have $(\ell, m) = (2, 3)$, and the assertions follow  from the remark above. 
   We may now assume that $m > 3$ and that both assertions are true for Schubert subvarieties $\Omega _\a(\ell', V_{m-1}, \B)$ for any vector space $V_{m-1}$ of dimension $m-1$ and for every $\ell'$ with $2 \le \ell' \le  m-1$. Note that in the case when $\ell' = 1$, the result is always true thanks to the Remark \ref{trivialcases} (0) above.
   Let $V_m$ be a vector space of dimension $m$ over $\Fq$ with a fixed ordered basis $\B = \{v_1, \dots, v_m\}$. For $\a = (\a_1, \dots, \a_\ell) \in I(\ell, m)$, we consider the Schubert subvariety $\Omb$. In view of Remark \ref{trivialcases}, we may assume that $\a_1, \dots, \a_\ell$ are not completely consecutive. 
  If $\ell = m-1$, then the assertion follows from Remark \ref{trivialcases} (II). In fact, the remark shows that the result is true even without the hypothesis that $q > q_0 (\ell)$. Thus, we may assume that $2 \le \ell \le m-2$. 

  We note that the function $q_0 (\ell)$ is increasing in $\ell$. Thus, whenever $q > q_0 (\ell)$, we have $q > q_0(\ell -1)$, which facilitates the induction in the following steps of the proof. 
  
  Suppose $\Pi_\a$ is a hyperplane of $\PP^{k_\a - 1}$. We distinguish the proof in two cases:
  \begin{enumerate}
      \item[{\bf Case 1:}] Suppose $\Pi_\a$ does not contain $\Omega_{\a'} (\ell, W_{m-1}, \B_1)$ for any $\Omega_{\a'} (\ell, W_{m-1}, \B_1) \in \mathfrak{S}_\a (\ell, V_m, \B)$. In view of Remark \ref{salpha} (b) and induction hypothesis, we have, for $q > q_0 (\ell)$ 
     $$|\Omega_{\a'} (\ell, W_{m-1}, \B_1) \cap \Pi_\a| \le |\Omega_{\a'} (\ell, W_{m-1}, \B_1)| - q^{\delta ({\a'})}.$$
    Since this is true for all elements of $\mathfrak{S}_\a(\ell,V_m,\B)$, we may assume that $M_{\a} \le |\Omega_{\a'} (\ell, W_{m-1} , \B_1)| - q^{\delta ({\a'})}$ as in Lemma \ref{dc} and obtain 
    $$|\Omb \cap \Pi_\a| \le \left(\frac{q^{m-\a_k} - 1}{q^{(m - \a_k) - (\ell - k)} - 1}\right)\left(|\Omega_{\a'} (\ell, W_{m-1}, \B_1)| - q^{\delta ({\a'})}\right).$$
    Since $q > q_{0} (\ell)$, using Lemma \ref{aux}, we have $$|\Omb \cap \Pi_\a| < |\Omb| - q^{\delta (\a)},$$
    as desired. We note that the equality does not occur in this case. 
      \item[{\bf Case 2:}] Suppose $\Pi_\a$ contains $\Omega_{\a'} (\ell, W_{m-1}, \B_1)$ for some element in $\mathfrak{S}_\a(\ell,V_m,\B)$. Proposition \ref{case1} applies, and yields
      $$|\Omb \cap \Pi_\a| \le |\Omega_{\a'}(\ell, W_{m-1}, \B_1)| + q^{m-\ell} e_{\widecheck{\a}}(\ell - 1, m-1).$$
      Since, by assumption $q > q_0 (\ell)$ and $q_0 (\ell) > q_0 (\ell - 1)$, the induction hypothesis implies (when $\ell =2$, we do not use the induction hypothesis, but we invoke Remark \ref{trivialcases} (0)) that $e_{\widecheck{\a}}(\ell - 1, m-1) = |\Omega_{\widecheck{\a}} (\ell - 1, W_{m-1}, \B_1)| - q^{\delta (\widecheck{\a})}$. The desired inequality now follows from Proposition \ref{ineq:dec} (a).   
  \end{enumerate}
  Note that, a possible equality arises if and only if $\Pi_\a$ contains an element $\Omega_{\a'} (\ell, W_{m-1}, \B_1) \in \mathfrak{S}_\a(\ell, V_m, \B)$ and that for each string $s^{-1} (\nu)$  of $\Omb \setminus \Omega_{\a'} (\ell, W_{m-1}, \B_1)$ the equality $|s^{-1}(\nu) \cap \Pi_\a| = e_{\widecheck{\a}} (\ell - 1, m-1)$ holds. As noted in Remark \ref{salpha}, we have $\Omb \setminus \Omega_{\a'} (\ell, W_{m-1}, \B_1) = \Omega_{\a} (\ell, V_m, \tilde{\B}) \setminus \Omega_{\a'} (\ell, W_{m-1}, \tilde{\B}')$. Thus, we may use the bijection
  $s^{-1} (\nu) \longleftrightarrow \Omega_{\widecheck{\a}} (\ell-1, W_{m-1}, \tilde{\B}')$ as discussed before.  By induction hypothesis (and Remark \ref{trivialcases} (0), in the case $\ell =2$), it follows that $\Pi_{\widecheck{\a}}$ is Schubert decomposable and hence decomposable.

We claim that the hyperplanes $\Pi_\a$ and $\Pi_{\widecheck{\a}}$ are given by the 
  same element of $\PP\left(\bigwedge^{m-\ell} V_m\right)$. Suppose that $\Pi_\a$ is 
  given by the polynomial $F_\a = \sum_{\beta \in \nabla (\a) \cap \Delta (\a')} 
  c_\beta X_{\beta}$. In view of \eqref{hypid}, the corresponding element of 
  $\bigwedge^{m-\ell} V_m$ is
  $$F_\a = \sum_{\beta \in \nabla (\a) \cap \Delta (\a')} 
  \epsilon (\beta) c_{\beta} v_{\beta^{\mathsf{C}}},$$
  where, for $\beta = (\beta_1, \dots, \beta_\ell) \in \nabla (\a) \cap \Delta (\a')$, 
  we write $\beta^{\mathsf{C}} = (\beta'_1, \dots, \beta'_{m-\ell})$ with 
  $\{\beta_1, \dots, \beta_\ell\} \cup \{\beta'_1, \dots, \beta'_{m-\ell}\} = 
  \{1, \dots, m\}$. Since $\beta_\ell = m$ for every 
  $\beta \in \nabla (\a) \cap \Delta (\a')$, we see that $\beta'_j < m$ for all 
  $j = 1, \dots, m-\ell$. 

  On the other hand, the hyperplane $\Pi_{\widecheck{\a}}$ is given by 
  $F_{\widecheck{\a}} = \sum_{\beta \in \nabla (\a) \cap \Delta (\a')} c_\beta 
  X_{\widecheck{\beta}}$, where $\widecheck{\beta} = (\beta_1, \dots, \beta_{\ell -1})$. 
  Applying \eqref{hypid} at the level of $\bigwedge^{(m-1)-(\ell-1)} V_{m-1}$, the 
  corresponding element is
  $$F_{\widecheck{\a}} = \sum_{\beta \in \nabla (\a) \cap \Delta (\a')} 
  \epsilon ({\widecheck{\beta}}) c_{\beta} v_{{\widecheck{\beta}}^{\mathsf{C}}},$$
  where ${\widecheck{\beta}}^{\mathsf{C}}$ denotes the complement of 
  $\widecheck{\beta}$ in $\{1, \dots, m-1\}$. Since $\beta_\ell = m$, the tuples 
  $\beta^{\mathsf{C}}$ and ${\widecheck{\beta}}^{\mathsf{C}}$ coincide, and 
  consequently $v_{\beta^{\mathsf{C}}} = v_{{\widecheck{\beta}}^{\mathsf{C}}}$. 
  Moreover, the juxtaposition $(\beta^{\mathsf{C}}, \beta)$ is obtained from 
  $({\widecheck{\beta}}^{\mathsf{C}}, \widecheck{\beta})$ by appending the entry $m$ 
  at the end; as $m$ is the largest element of $\{1, \dots, m\}$, the permutation 
  sorting $(\beta^{\mathsf{C}}, \beta)$ into $(1, \dots, m)$ is the same as the one 
  sorting $({\widecheck{\beta}}^{\mathsf{C}}, \widecheck{\beta})$ into 
  $(1, \dots, m-1)$. Hence $\epsilon (\beta) = \epsilon ({\widecheck{\beta}})$ for 
  every $\beta \in \nabla (\a) \cap \Delta (\a')$, and therefore 
  $F_\a = F_{\widecheck{\a}}$ as elements of $\bigwedge^{m-\ell} V_m$. Thus our claim 
  is proved.

Since $\Pi_{\widecheck{\a}}$ is decomposable, there exists a decomposable element 
  $z \in \bigwedge^{m-\ell} V_{m-1}$ such that $z_{\nabla (\widecheck{\a})} = F_{\widecheck{\a}}$. 
  As the wedge factors of $z$ lie in $V_{m-1} \subset V_m$, the element $z$ remains 
  decomposable when regarded as an element of $\bigwedge^{m - \ell} V_m$. The claim 
  proved above shows that $z_{\nabla (\a)} = F_\a$, and hence $\Pi_\a$ is decomposable. 
  Consequently, by \cite[Theorem 5.5]{GS}, as pointed out in Remark \ref{where} (a), 
  the hyperplane $\Pi_\a$ is Schubert decomposable. 
\end{proof}

We conclude this article with a remark on the quantity $q_0 (\ell)$.

\begin{remark}\normalfont
    The function $q_0 (\ell)$, as pointed out before, is an increasing function on $\ell$. One can check with a calculator that $q_0 (2) = 2$, $q_0 (3) \approx 3.414, ... $. In particular, our result establishes the MWCC for $\ell = 2$ with $q \ge 3$, $\ell = 3$ with $q \ge 4$ etc. We also note, and leave it to the reader to check that,
    $$\lim_{\ell \to \infty} \dfrac{q_0 (\ell)}{\ell} = \frac{1}{\ln 2} \approx 1.4427.$$
    An obvious question arises. How can one improve the situation? In particular, is it possible to use our methodology to prove the MWCC in complete generality? If not, is it at least possible to replace $q_0 (\ell)$ with some $q_1 (\ell)$ in our hypothesis of Theorem \ref{main} where $q_1 (\ell) < q_0(\ell)$ for all large values of $\ell$? We do not yet have answers to these questions. However, it is conceivable that the possible improvement lies in obtaining better bounds for the number of points on Schubert subvarieties over finite fields in Proposition \ref{upp}, which might also lead to understanding the coefficients $a_i$ in \eqref{schen1}. We leave these as problems for further research.  
\end{remark}

\textbf{Acknowledgment}
The authors are grateful to the anonymous referees for their careful reading of the article and for their suggestions, which have significantly improved the exposition.

\end{document}